\documentclass[10pt]{amsart}

\usepackage[utf8]{inputenc} 
\renewcommand\labelenumi{(\roman{enumi})}
\renewcommand\theenumi\labelenumi

\usepackage{amsmath}
\usepackage{amstext}
\usepackage{amsfonts}
\usepackage{amsthm}
\usepackage{caption}
\usepackage{amssymb}
\usepackage{comment}
\usepackage[normalem]{ulem}
\usepackage{enumerate}

\usepackage{graphicx}
\usepackage{subfigure}

\usepackage{color}
  
 \definecolor{darkgreen}{rgb}{0,.6,0}
 
\newcommand{\R}{{\mathbb{R}}}
\newcommand{\E}{{\mathbb{E}}}
\newcommand{\N}{{\mathbb{N}}}
\newcommand{\IT}{\mathbb{T}}

\newcommand{\cL}{{\mathcal{L}}} 
\newcommand{\cH}{\mathcal{H}}
\newcommand{\cV}{{\mathcal{V}}}
\newcommand{\cI}{\mathcal{I}}
\newcommand{\cJ}{\mathcal{J}}
\newcommand{\cK}{\mathcal{K}}
\newcommand{\cD}{\mathcal{D}}
\newcommand{\cO}{\mathcal{O}}

\renewcommand{\P}{{\mathbf{P}}} 

\newcommand{\B}{{\mathbf{B}}}
\newcommand{\LB}{\cL}

\newcommand{\diff}[1]{\,\dd#1}
\newcommand{\dd}{\mathrm{d}}

\newcommand{\norm}[2][]{\|#2\|_{#1}}

\newcommand{\Cs}{\mathcal{C}_{\mathrm{s}}}
\newcommand{\cC}{\mathcal{C}}

\newcommand{\triple}{{\vert\kern-0.25ex\vert\kern-0.25ex\vert}}

\newcommand{\Id}{\mathrm{id}_H}

\newcommand{\bA}{{\mathbf{A}}}
\newcommand{\bB}{{\mathbf{B}}}

\newcommand{\bG}{{\mathbf{G}}}

\newcommand{\bL}{{\mathbf{L}}}
\newcommand{\bM}{{\mathbf{M}}}
\newcommand{\bR}{{\mathbf{R}}}

\newcommand{\bx}{{\mathbf{x}}}

\theoremstyle{definition}
\newtheorem{definition}{Definition}[section]
\newtheorem{assumption}[definition]{Assumption}

\newtheorem{example}[definition]{Example}

\theoremstyle{plain}
\newtheorem{theorem}[definition]{Theorem}
\newtheorem{lemma}[definition]{Lemma}
\newtheorem{corollary}[definition]{Corollary}
\newtheorem{prop}[definition]{Proposition}

\makeatletter
\@namedef{subjclassname@1991}{{\normalfont 2020} Mathematics Subject Classification}
\makeatother

\usepackage[pdftex,
            pdftitle={Finite element approximation of Lyapunov equations for the computation of quadratic functionals of SPDEs},
            pdfauthor={A. Andersson, A. Lang, A. Petersson, and L. Schroer},
            bookmarksopen,
            colorlinks,
            linkcolor=black,
            urlcolor=black,
	    citecolor=black
]{hyperref}

\title[Approximation of Lyapunov equations]{Finite element approximation of Lyapunov equations related to parabolic stochastic PDEs}

\author[A.~Andersson]{Adam Andersson} \address[Adam Andersson]{\newline Department of Mathematical Sciences
	\newline Chalmers University of Technology \& University of Gothenburg
	\newline S--412 96 G\"oteborg, Sweden.
	\newline Saab AB Radar Solutions
        \newline S--412 76 G\"oteborg, Sweden} \email[]{adam.andersson@chalmers.se, adam.andersson2@saabgroup.com}	
\author[A.~Lang]{Annika Lang} \address[Annika Lang]{\newline Department of Mathematical Sciences
	\newline Chalmers University of Technology \& University of Gothenburg
	\newline S--412 96 G\"oteborg, Sweden.} \email[]{annika.lang@chalmers.se}
\author[A.~Petersson]{Andreas Petersson} \address[Andreas Petersson]{\newline Department of Mathematics
\newline University of Oslo
\newline Postboks 1053
\newline Blindern
\newline 0316 Oslo, Norway.} 
\email[]{andreep@math.uio.no}
\author[L. Schroer]{Leander Schroer} \address[Leander Schroer]{\newline Sopra Steria SE
\newline Sprengelstrasse 40
\newline 13353 Berlin, Germany} 
\begin{document}

\keywords{Lyapunov equations, finite element method, stochastic partial differential equations, stochastic heat equation, weak convergence, parabolic Anderson model, numerical approximation, multiplicative noise}

\subjclass[1991]{65M60, 60H15, 65J10, 65C30, 60H35, 49J20}
	
 \thanks{The authors wish to express many thanks to Stig Larsson for fruitful discussions and helpful comments. This work was supported in part by the Chalmers AI Research Centre (CHAIR), by the Knut and Alice Wallenberg foundation, by the Research Council of Norway (RCN) through project no.\ 274410, by the Swedish Research Council under Reg.~No.~621-2014-3995 and project no.\ 2020-04170 and by the Wallenberg AI, Autonomous Systems and Software Program (WASP) funded by the Knut and Alice Wallenberg Foundation. The authors have no relevant financial or non-financial interests to disclose.} 

\begin{abstract}
A numerical analysis for the fully discrete approximation of an operator Lyapunov equation related to linear SPDEs (stochastic partial differential equations) driven by multiplicative noise is considered. The discretization of the Lyapunov equation in space is given by finite elements and in time by a semiimplicit Euler scheme. The main result is the derivation of the rate of convergence in operator norm. Moreover, it is shown that the solution of the equation provides a representation of a quadratic and path dependent functional of the SPDE solution. This fact yields a deterministic numerical method to compute such functionals. As a secondary result, weak error rates are established for a fully discrete finite element approximation of the SPDE with respect to this functional. This is obtained as a consequence of the approximation analysis of the Lyapunov equation. It is the first weak convergence analysis for fully discrete finite element approximations of SPDEs driven by multiplicative noise that obtains double the strong rate of convergence, especially for path dependent functionals and smooth spatial noise. Numerical experiments illustrate the results empirically and it is demonstrated that the deterministic method has advantages over Monte Carlo sampling in a stability context.
\end{abstract}

\maketitle

\section{Introduction}

Lyapunov and Riccati equations have been studied for linear quadratic control and filtering of stochastic partial differential equations (SPDEs for short) since the late 1970s (see references below). Riccati equations are operator equations containing a nonlinear quadratic term. Their solutions provide optimal feedback controls for stochastic control problems and the covariance operators of the filtering distribution in optimal filtering with the Kalman--Bucy filter (see, e.g., \cite{falb1967} for a Hilbert-space-valued setting with bounded generators and trace-class noise). Removing the quadratic term, linear operator equations called Lyapunov equations are obtained. They are crucial for stability analysis and connected to quadratic functionals of SPDEs driven by multiplicative noise. 

In this work we establish a complete error analysis for numerical discretizations of Lyapunov equations by the semiimplicit Euler method in time and a finite element method in space. We connect these approximations to approximations of path dependent quadratic functionals of SPDEs. This connection allows us to show weak convergence rates of the SPDE approximation which are twice the strong rates. Our analysis can be used as a stepping stone for approximation results on Riccati equations in future work.

The two main equations that we connect are the linear parabolic SPDE
\begin{equation}\label{eq:SPDE}
\dd X(t) + AX(t) \, \dd t = B(X(t)) \, \dd W(t)
\end{equation}
with initial condition $X(0) = X_0$ in a Hilbert space~$H$ and
the operator valued Lyapunov equation, written in variational form,
\begin{equation}\label{eq:Lyapunov_var}
\frac{\dd}{\dd t}
\langle
L(t) \phi,\psi
\rangle
+
a(L(t)\phi,\psi)
+
a(L(t)\psi,\phi)
=
\langle
R\phi,R\psi
\rangle
+
\langle
L(t)B(\phi),B(\psi)
\rangle_{\LB_2^0}
\end{equation}
with $L(0) = G^\ast G$, where $\LB_2^0$ refers to a class of Hilbert--Schmidt operators.
We show in Section~\ref{sec:Lyapunov_equation_and_SPDE} that these are connected via
\begin{align}\label{eq:polarization_identity}
\langle L(T)x,x \rangle =\Phi(x)
\end{align} 
with respect to the quadratic functional
\begin{equation}\label{eq:Phi}
\Phi(x)
= \E \Big[ \int_0^T \|R X(t)\|^2  \, \dd t
+ \|G X(T)\|^2
\Big|
X(0)=x
\Big].
\end{equation}
Here $a(\cdot,\cdot)$ denotes the bilinear form corresponding to the operator~$-A$ that generates an analytic semigroup and $W$ is a cylindrical Wiener process. This includes the classical case of a $Q$-Wiener process with trace-class covariance. For the complete details on the setting, the reader is referred to Section~\ref{sec:setting}.

We study Lyapunov equations in a new generality suitable for numerical analysis with an emphasis on regularity. It was surprising to us that the literature does not cover the setting of cylindrical noise (cf.\ \cite{daprato1984b,Flandoli1982, Flandoli1986, Flandoli1990a,Tessitore2005, hutang2018,Ichikawa1979,Tessitore1992b}). We therefore develop the solution theory for the Lyapunov equation and prove existence and uniqueness by the Banach fixed point theorem and the Gronwall lemma. 

To show~\eqref{eq:polarization_identity}, we use tools from numerical analysis. We approximate both equations \eqref{eq:Lyapunov_var} and~\eqref{eq:SPDE} on a finite-dimensional subspace~$V_h$ such as a finite element space and show that \eqref{eq:polarization_identity} holds in the semidiscrete setting. Convergence establishes equality in the limit and gives as a byproduct convergence to the Lyapunov equation and weak convergence of the SPDE approximation.

For the fully discrete approximation of the Lyapunov equation~\eqref{eq:Lyapunov_var}, we discretize the above semidiscrete approximation by a semiimplicit Euler method in time. Results on numerical methods for Lyapunov and Riccati equations for stochastic problems are rare. The results of this paper are most closely related to those of~\cite{mena2017}, which only considers one-dimensional noise in an abstract approximation framework for Riccati equations. Connected to our problem are also~\cite{kuehnkurschner2020}, in which a time-independent Lyapunov equation related to an approximation of~\eqref{eq:SPDE} is employed as part of a bigger problem, and~\cite{bennerstillfjordtrautwein2021}, which assumes convergence of an approximation of a Riccati equation to derive strong convergence of a finite element approximation of a controlled version of~\eqref{eq:SPDE}. To the best of our knowledge, this work is the first to provide rigorous a priori convergence rates for a fully discrete numerical approximation of the Lyapunov equation~\eqref{eq:Lyapunov_var} in the infinite-dimensional noise setting and the first to connect such approximations to weak convergence for the related SPDE~\eqref{eq:SPDE}.

Weak convergence of numerical approximations of SPDEs with additive noise is a well understood topic, see, e.g., for implicit Euler in time~\cite{WangGan} and for finite element and spectral Galerkin methods in space \cite{AnderssonLarsson2016,Brehier3,cuihong2018}. For multiplicative noise the literature is still restricted to special cases. Weak rates of convergence have been obtained for discretization in time with implicit \cite{BrehierDebussche2017,debussche2011} and exponential \cite{Kurniawan2016} Euler schemes and in space with a spectral Galerkin method~\cite{conus2014}. For the finite element method, proofs are restricted to the spatially semidiscrete setting with (essentially) linear multiplicative space-time white noise~\cite{AnderssonLarsson2016}. The fully discrete setting and more regular noise are still open. One reason for this is the appearance of an extra term~\cite{AnderssonLarsson2016} which is not present in the spectral Galerkin method~\cite{conus2014}.

Based on the convergence analysis of the fully discrete approximation of the Lyapunov equation and the connection~\eqref{eq:polarization_identity}, we are able to extend the existing weak convergence analysis for finite element approximations. In the fully discrete setting of \eqref{eq:Lyapunov_var} and~\eqref{eq:SPDE}, we establish \eqref{eq:polarization_identity} up to a small error and are thus the first to show convergence for path dependent quadratic functionals with multiplicative white or colored noise in the finite element setting. The rate is twice that of strong convergence and coincides with that for additive noise.

Our numerical schemes for~\eqref{eq:SPDE} and~\eqref{eq:Lyapunov_var} and their convergence open up for two methods to approximate~\eqref{eq:Phi}: either deterministically for all initial conditions~$x$ with~\eqref{eq:Lyapunov_var} or combining~\eqref{eq:SPDE} with a Monte Carlo method. Depending on the application one or the other method might be more suitable. If the problem at hand is the computation of~\eqref{eq:Phi} with respect to all initial conditions in parallel, our Lyapunov method is preferable. This method also has an advantage
\begin{enumerate}
\item
if the operator $R$ is non-local, since then multiplication with a dense matrix needs to be repeated for each time step and each sample in a Monte Carlo simulation. In a Lyapunov method, a similar dense matrix operation only needs to be repeated once for each time step.
\item \label{enum:stability}
under multiplicative noise of large magnitude, since this causes stability problems~\cite{thalhammer2017}. More precisely, the zero solution $X=0$ can be asymptotically stable in the almost sure sense but asymptotically mean square unstable, simultaneously. In this setting, the Monte Carlo method fails to approximate $\Phi$ while our deterministic Lyapunov method faces no problem. We demonstrate this phenomenon in an example in Section~\ref{sec:numerics}.
\end{enumerate}

The manuscript is organized as follows. In Section~\ref{sec:setting} the abstract setting and notation of the paper are introduced along with assumptions on the family of approximation spaces~$(V_h)_{h\in(0,1]}$. Existence and uniqueness of a mild solution to~\eqref{eq:Lyapunov_var} and its spatial and temporal regularity are established in Section~\ref{sec:Lyapunov_equation_and_SPDE}. Furthermore, \eqref{eq:polarization_identity} is shown via an analogous equality in the semidiscrete setting, i.e., \eqref{eq:SPDE} and~\eqref{eq:Lyapunov_var} are solved on~$V_h$. Section~\ref{sec:fully_discrete_Lyapunov} and Section~\ref{sec:fully_discrete_spde} are devoted to convergence analyses of fully discrete semiimplicit approximation schemes for~\eqref{eq:Lyapunov_var} and~\eqref{eq:SPDE}, respectively. 
In Section~\ref{sec:numerics}, numerical experiments conclude the manuscript that illustrate the theoretical results and compare the deterministic approach via~\eqref{eq:Lyapunov_var} with a Monte Carlo simulation of~\eqref{eq:SPDE} with respect to the stability issues named in~\ref{enum:stability} above. For completeness we include proofs based on standard arguments in the appendix.

\section{Notation and abstract setting}\label{sec:setting}

We start by introducing the necessary notation. For separable Hilbert spaces $(U, \langle \cdot,\cdot \rangle_U)$ and $(V, \langle \cdot,\cdot \rangle_V)$ with corresponding norms, we denote by $\LB(U,V)$ the Banach space of all bounded linear operators $U\to V$ equipped with the operator norm, where we abbreviate $\LB(U)=\LB(U,U)$. The space $\Sigma(U)\subset \LB(U)$ is the closed subspace of all self-adjoint operators and $\Sigma^+(U)\subset \Sigma(U)$ is the restriction to all operators that are additionally non-negative definite. By $\LB_2(U,V)\subset\LB(U,V)$ we denote the space of Hilbert--Schmidt operators $U\to V$. This is a Hilbert space with norm and inner product given by
\begin{equation*}
  \|T\|_{\LB_2(U,V)}^2
  =
    \sum_{i \in \N}
      \|Te_i\|_V^2
  ,\qquad
  \langle
    T,S
  \rangle_{\LB_2(U,V)}
  =
  \sum_{i \in \N}
  \langle
    T e_i,S e_i
  \rangle,
\end{equation*}
where  $(e_i)_{i=1}^\infty$ is an orthonormal basis of~$U$. The definition is independent of the choice of basis. For an interval $I\subset \R$, we denote by $\cC(I,\LB(U))$ and $\Cs(I,\LB(U))$ the spaces of continuous and strongly continuous functions from $I$ to $\LB(U)$, respectively. 

The beta function $\B \colon (0,\infty)\times(0,\infty)\to\R$ is given by $\B(x,y) = \int_0^1t^{x-1}(1-t)^{y-1}\diff t$. By a change of variable the following very useful identity is obtained: For all $t_1 \leq t_2$, $x,y\in (0,\infty)$,
\begin{equation}\label{eq:beta_integral}
  \int_{t_1}^{t_2}
    (s-t_1)^{x-1}(t_2-s)^{y-1}
  \diff s
  =
  \B(x,y)\, |t_2-t_1|^{x+y-1}.
\end{equation}

We next introduce the setting that we consider throughout the article. Here $U$ and $H$ are fixed separable Hilbert spaces and by $\langle\cdot,\cdot\rangle$ and $\|\cdot\|$ we denote the inner product of $H$ and its induced norm, respectively.
\begin{assumption}\label{ass:SPDE}
	Equations~\eqref{eq:SPDE} and~\eqref{eq:Phi} satisfy the following conditions:
	\begin{enumerate}
		\item \label{ass:SPDE_A} The linear operator $A: \cD(A) \subset H \rightarrow H$ is densely defined, self-adjoint and positive definite with compact inverse.
		\item \label{ass:SPDE_W} The process $W = (W(t))_{t \in \IT}$ is an adapted cylindrical $I_U$-Wiener process on a filtered probability space $(\Omega,\mathcal F, (\mathcal F_t)_{t\in \IT},\P)$.
		\item \label{ass:SPDE_B_beta} For a fixed regularity parameter $\beta \in (0,1]$, the linear operator $B$ satisfies $\|A^{(\beta-1)/2}B\|_{\LB(H,\LB_2(U,H))} < \infty$.
		\item \label{ass:SPDE_R_G} The linear operators $R$ and $G$ satisfy $\|R\|_{\LB(H)}< \infty$ and $\|G\|_{\LB(H)}< \infty$.
	\end{enumerate}
\end{assumption}

Fractional powers $(A^{r/2})_{r\in\R}$ of~$A$, such as $A^{(\beta-1)/2}$ in the assumption above, are well-defined and enable us to define the spaces $(\dot H^r)_{r\in\R}$, which are used to measure spatial regularity. More specifically, for $r\geq0$
\begin{equation*}
\dot{H}^r
= \{ \phi \in H, \|\phi\|_{\dot H^r} =\|A^{\frac{r}{2}}\phi \| < \infty \}
\end{equation*}
and for $r<0$ the space $\dot H^{r}$ is the closure of $H$ under the $\|A^{r/2}\cdot\|$-norm and  $\dot H^r = (\dot H^{-r})'$, the dual space of $\dot H^{-r}$ with respect to $\langle \cdot,\cdot \rangle$. In that way we obtain a family $(\dot H^r)_{r\in\R}$ of separable Hilbert spaces with the property that $\dot{H}^r \subset \dot{H}^s$ whenever $r \ge s \in \R$, where the embedding is dense and continuous. Moreover, by \cite[Lemma~2.1]{BKK20}, for every $s \in \R$,  $A^{r/2}$ can be uniquely extended to an operator in $\LB(\dot{H}^s,\dot{H}^{s-r})$. We make no notational distinction between $A^{r/2}$ and its extension and define the corresponding bilinear form $a: \dot H^1 \times \dot H^1 \rightarrow \R$ for $\phi, \psi \in \dot H^1$ by
\begin{equation}\label{eq:aA}
a(\phi, \psi) = \langle A^{\frac{1}{2}} \phi, A^{\frac{1}{2}} \psi \rangle.
\end{equation}

The operator $-A$ is the generator of an analytic semigroup $S=(S(t))_{t \ge 0}$ of bounded linear operators on~$H$ that extends to $\dot{H}^r$, $r<0$.  As for $A$, we do not differentiate between the semigroup~$S$ and its extension. 
The analyticity of the semigroup implies the existence of constants $(C_\theta)_{\theta\geq0}$ such that for all $\theta\in[0,\infty)$
\begin{equation}\label{eq:S_smooth}
  \sup_{t>0}
  t^{\frac\theta2}
  \big\|
    A^{\frac\theta2}S(t)
  \big\|_{\LB(H)}
  \leq
  C_\theta
\end{equation}
and for all $\theta\in[0,2]$
\begin{equation}\label{eq:S_Holder1}
\sup_{t>0}
t^{-\frac\theta2}
  \big\|
    A^{-\frac\theta2} 
    \big(
      S(t) - \Id
    \big)
  \big\|_{\LB(H)}
  \leq 
  C_\theta.
\end{equation}
These regularity estimates play an essential role in our proofs. We refer to \cite[Appendix~B]{kruse2013} for a detailed introduction to this setting.

Assumption~\ref{ass:SPDE}\ref{ass:SPDE_W} on~$W$ includes white noise in~$H$ (by letting $U=H$) as well as $H$-valued trace-class $Q$-Wiener processes (by letting $U=Q^{1/2}(H)$, cf.\ \cite[Theorem~7.13]{Peszat07}). We introduce the notation $\LB_2^0 = \LB_2(U,H)$ and set $\IT=[0,T]$ for $T>0$. Note that for predictable stochastic processes $\Psi\in L^2(\IT\times\Omega;\LB_2^0)$ the stochastic integral $\int_0^T \Psi(t)\diff W(t)\in L^2(\Omega;H)$ is well-defined.

We are now in place to introduce the setting for the SPDE~\eqref{eq:SPDE}. By \cite[Theorem 2.9]{AJK21} \eqref{eq:SPDE} admits an up to modification unique mild solution, i.e., a predictable process $X\colon\IT \times\Omega\to H$ that satisfies for all $t\in\IT$, $\P$-a.s.
\begin{equation}
\label{eq:SPDE_mild}
X(t) = S(t) X_0 + \int_0^t S(t-s) B(X(s)) \, \dd W(s)
\end{equation}
and
\begin{equation}\label{eq:X_moment}
\sup_{t\in\IT}\|X(t)\|_{L^2(\Omega;H)}
\lesssim \|X_0\|<\infty,
\end{equation} 
where we denote $a \lesssim b$ if there exists a generic constant $C$ such that $a \le C b$ and the size of the constant is of minor relevance.

Next, we introduce spatial approximation spaces. Let $(V_h)_{h \in (0,1]}$ be a family of finite-dimensional subspaces of $\dot{H}^1$, where $h$ denotes the refinement parameter. We equip $V_h$ with the same inner product as $H$ so that for an operator $T \in \LB(V_h)$,
\begin{equation*}
\big\| T \big\|_{\LB(V_h)} = \big\| T P_h \big\|_{\LB(H)}.
\end{equation*} 
Here $P_h: \dot{H}^{-1} \to V_h$ is the generalized orthogonal projector (see, e.g., \cite[Section~3.2]{kruse2013}) which coincides with the standard orthogonal projector when restricted to $H$. Let $A_h: V_h \to V_h$ be the unique operator defined for $\phi_h,\psi_h \in V_h$ by
\begin{equation*}
\langle
A_h\phi_h,\psi_h
\rangle
=
a(\phi_h,\psi_h).
\end{equation*}
This implies that $A_h$ is self-adjoint and positive definite on $V_h$. Therefore, $-A_h$ generates an analytic semigroup $S_h: [0,\infty) \to \LB(V_h)$ on $V_h$ and fractional powers of $A_h$ are defined in the same way as for~$A$. For brevity we write $A^{\theta/2}_h$ for $A^{\theta/2}_h P_h$ and $S_h(t)$ for $S_h(t) P_h$, $\theta \in \R$, $t \in [0,T]$. By \cite[(3.12)]{kruse2013} and \cite[Lemma~B.9(ii)]{kruse2013} there exist constants $(D_\theta)_{\theta\geq0}$ so that for all $\theta\geq0$
\begin{equation}
\label{eq:Analytic}
\sup_{h\in(0,1], t > 0}
t^{\frac\theta2}
\|A_h^{\frac\theta2}S_h(t)\|_{\LB(H)}
\leq
D_\theta
\end{equation}
and for all $\theta\in[0,2]$ that
\begin{equation}
\label{eq:Analytic2}
\sup_{h\in(0,1], t > 0}
t^{-\frac\theta2} \|A_h^{-\frac\theta2}(S_h(t)-P_h)\|_{\LB(H)}
\leq
D_\theta.
\end{equation}
Here and below we use the notation $D_\theta$ for a constant depending on the choice of discretization but not the specific value of $h \in (0,1]$. The optimal value may differ from line to line.

To guarantee that $(V_h)_{h\in (0,1]}$ has appropriate approximation properties and includes finite element approximations, we make the following assumptions. 

\begin{assumption}\label{ass:Ah}
There exist constants $(D_\theta)_{\theta\geq-1}$ such that
\begin{enumerate}
	\item \label{eq:Ritz_assumption} 
	for $\theta \in \{1,2\}$, $h \in (0,1]$:  $\big\|
	(A_h^{-1} P_h A -\Id) A^{-\theta/2}
	\big\|_{\LB(H)}
	\leq D_\theta h^\theta$, 
	\item \label{eq:Ph}
	for $\theta\in[0,2]$, $h \in (0,1]$: $\| A^{-\theta/2}(P_h-\Id)\|_{\LB(H)} \, \leq D_\theta h^{\theta}$,
	\item \label{eq:inv_ineq}
	for $\theta \in [0,2]$, $h \in (0,1]$: $\big\|A_h^{\theta/2} \big\|_{\LB(H)}
	\leq
	D_\theta h^{-\theta}$ and 
	\item \label{eq:AAh}
	for $\theta \in [-1,1]$, $\phi \in \dot H^{\theta}$: $\sup_{h\in(0,1]}\big\|
	A_h^{\theta/2} \phi
	\big\|
	\leq
	D_\theta
	\big\|
	A^{\theta/2}\phi
	\big\|$ . 
\end{enumerate}
\end{assumption}

\begin{example}
	\label{ex:heat_equation}
	Assumption~\ref{ass:Ah} holds in the following finite element setting. Let $H=L^2(D)$ for some bounded, convex polygonal domain $D \subset \R^d$, $d \in \{1,2,3\}$ and $A = - \Delta$ denote the Laplace operator with zero Dirichlet boundary conditions. Let $(\mathcal{T}_h)_{h\in(0,1]}$ be a regular family of triangulations of $D$ and let $V_h$ be the space of all continuous functions that are piecewise polynomials of some fixed degree on~$\mathcal{T}_h$. Then \ref{eq:Ritz_assumption} and \ref{eq:Ph} hold true, see, e.g., \cite[Chapters 1-3]{thomee2006}. If we assume in addition to this that the family $(\mathcal{T}_h)_{h\in(0,1]}$ is quasi-uniform, then we also have \ref{eq:inv_ineq} and \ref{eq:AAh}, see, e.g., \cite[(3.28)]{thomee2006} and \cite{crouzeix1987}. 
\end{example}

A consequence of~\ref{eq:AAh} and the definition of $A_h$ (see~\cite[page~1341]{AnderssonLarsson2016}) is the existence of constants $(D_\theta)_{\theta\geq-1}$ such that for all $\theta \in [-1,1]$
 	\begin{equation}\label{eq:AhA}
	\sup_{h\in(0,1)}
	\big\|
	A^{\frac{\theta}2}A_h^{^{-\frac{\theta}2}}
	\big\|_{\LB(H)}
	\le
	D_\theta \sup_{h\in(0,1)} \big\|
	A_h^{\frac{\theta}2}A_h^{^{-\frac{\theta}2}}
	\big\|_{\LB(H)}	=D_\theta.
	\end{equation}
Using also \ref{eq:Ritz_assumption}  and \ref{eq:inv_ineq} one can show (cf.\ the proof of \cite[Theorem~4.4]{larsson2011}) the existence of constants $(D_\theta)_{\theta\geq-1}$ such that for all $\theta \in [-1,2]$
\begin{equation}\label{eq:AhA2}
\sup_{h\in(0,1)}
\big\|
A_h^{\frac{\theta}2}
A^{-\frac{\theta}2}
\big\|_{\LB(H)}
\leq
D_\theta
\big\|
A^{\frac{\theta}2}A^{^{-\frac{\theta}2}}
\big\|_{\LB(H)}
=
D_\theta.
\end{equation}

Let $E_h: (0,1] \to \LB(H)$ denote the \emph{error operator} $E_h = S - S_h$. As another consequence of~\ref{eq:Ritz_assumption} we obtain that there exist constants $(D_\theta)_{\theta\geq0}$ such that for all $h \in (0,1]$, $\mu \in [0,2)$ and $\theta \in [0,1]$ with $\mu + \theta < 2$,
\begin{equation}
\label{eq:Eh}
\sup_{t>0} t^{\frac{\mu+\theta}2}\|E_h(t)A^{\frac{\theta}2}\|_{\LB(H)} = \sup_{t>0} t^{\frac{\mu+\theta}2} \|A^{\frac{\theta}2}E_h(t)\|_{\LB(H)}
\leq
D_\theta
h^\mu.
\end{equation}
This is proven analogously to \cite[Lemma~5.1]{AnderssonKruseLarsson}, replacing the use of \cite[Lemma~3.12]{kruse2013} with \cite[Lemma~3.8]{kruse2013}, using also \eqref{eq:AhA2} and the fact that $E_h(t)$ is self-adjoint for all $t \in \IT$. In the next sections, we frequently use this bound with $\mu = 2 \rho$.

We next introduce the setting for the full discretization in space and time. Recall that $\IT=[0,T]$ and set $\IT_0=(0,T]$ for $T>0$. For $\tau\in(0,1]$, let $(t_n)_{n\in \N_0}\subset\R$ be the uniform discretization of $\IT$ given by $t_n=\tau n$ and $N_\tau = \inf\{n\in\N : t_{n+1} \notin \IT\}$. Let us denote by $S_{h,\tau}$ the implicit Euler approximation of the semigroup at time $\tau$, i.e., $S_{h,\tau}=(P_h +\tau A_h)^{-1}$. The discrete family $(S_{h,\tau}^n)_{n\in\{0,\dots,N_\tau\}}$ of powers of $S_{h,\tau}$ acts as a fully discrete approximation of the semigroup~$S$. We again write, for brevity, $S_{h,\tau}^n$ for $S_{h,\tau}^n P_h$. 

Let us now collect three properties of the discrete approximation $S_{h,\tau}$ of the semigroup and the error operator $E_{h,\tau}^n = S_{h,\tau}^n-S_h({t_n})$. There exist constants $(D_\theta)_{\theta\geq0}$ such that for all $\theta\in [0,2]$ and $\tau \in (0,1]$
\begin{equation}
\label{eq:S_h_bound}
\sup_{h\in(0,1], n\in\{1,\dots,N_\tau\}}
t_n^{\frac{\theta}2} \|A_h^{\frac{\theta}2} S_{h,\tau}^n\|_{\LB(H)}\leq D_\theta
\end{equation}
for all $\theta\in [0,1]$, $\rho \in [0,2]$ and $\tau \in (0,1]$
\begin{equation}
\label{eq:E_h_bound}
\sup_{h\in(0,1], n\in\{1,\dots,N_\tau\}}
t_n^{\frac{\rho+\theta}2} \|A_h^{\frac{\theta}2}E_{h,\tau}^n\|_{\LB(H)}
\leq
D_\theta \tau^{\rho/2}
\end{equation}
and for all $\theta\in [0,1]$ and $\tau \in (0,1]$
\begin{equation}
\label{eq:S_h_Holder}
\sup_{h \in(0,1]}
\|A_h^{-\theta}(S_{h,\tau}-P_h)\|_{\LB(H)}
\leq
D_\theta 
\tau^{\theta}.
\end{equation}
For a proof of \eqref{eq:S_h_bound}, see, e.g., \cite[Lemma~7.3]{thomee2006}. We show \eqref{eq:E_h_bound} in Proposition~\ref{prop:AE} and the well-known result \eqref{eq:S_h_Holder} can be shown in a similar way, see, e.g., \cite[Lemma~B.9]{kruse2013}.

We use the abbreviations $b=\|A^{(\beta-1)/{2}}B\|_{\LB(H,\LB_2^0)} = \|B\|_{\LB(H,\LB_2(U,\dot H^{\beta-1}))}$, $r=\|R\|_{\LB(H)}$ and $g=\|G\|_{\LB(H)}$.

\section{Theory of the Lyapunov equation and the SPDE}
\label{sec:Lyapunov_equation_and_SPDE}

The goal of this section is threefold. We start with existence, uniqueness and regularity of the solution to the Lyapunov equation~\eqref{eq:Lyapunov_var} in Section~\ref{subsec:exLyapunovSol}. Second, we present in Section~\ref{subsec:semidiscreteApprox} an error analysis for semidiscrete space approximations of the Lyapunov equation~\eqref{eq:Lyapunov_var} and the SPDE~\eqref{eq:SPDE}. This is used in Section~\ref{subsec:LyapunovSolvesQoI} to show~\eqref{eq:polarization_identity}. As an immediate consequence we obtain weak convergence rates for the semidiscrete SPDE approximation to~\eqref{eq:SPDE}

\subsection{Existence, uniqueness and regularity}\label{subsec:exLyapunovSol}

While the variational form~\eqref{eq:Lyapunov_var} of the Lyapunov equation is natural for numerics, it is more natural to work in the semigroup framework for the regularity analysis. The mild form of the Lyapunov equation reads: Find $L\colon\IT\to\LB(H)$ such that for all $t\in\IT$ and $\phi\in H$
\begin{equation}\label{eq:Lyapunov_mild}
  L(t)\phi
=
  S(t)
  G^*G
  S(t)\phi
  +
  \int_0^t
    S(t-s)
    \big(
      R^*R
      +
      B^*L(s)B
    \big)
    S(t-s)\phi \,
  \diff s.
\end{equation}
We note that the mapping
$
  [0,t]\ni s
  \mapsto
  S(t-s)
  \big(
    R^*R
    +
    B^*L(s)B
  \big)
  S(t-s)
  \in \LB(H)
$ is not necessarily Bochner integrable due to the semigroup being only strongly measurable, which requires $\phi$ to be inside the integral. 

With some abuse of notation, we write $B^*$ for the operator in $\cL(\LB_2(U,\dot{H}^{1-\beta}),H)$ that for all $K \in \LB_2(U,\dot{H}^{1-\beta})$ and $v \in H$ satisfies
\begin{equation*}
	\langle B^* K, v \rangle = \langle K, B v \rangle_{\LB_2^0} = \langle A^{\frac{1-\beta}{2}} K, A^{\frac{\beta-1}{2}} B v \rangle_{\LB_2^0}.
\end{equation*}
It satisfies $\|B^*A^{(\beta-1)/2}\|_{\LB(\LB_2^0,H)} = \|A^{(\beta-1)/2} B \|_{\LB(H,\LB_2^0)} = b$.

Let $\cV$ be the space of all operator-valued functions $\Upsilon\colon\IT\to\LB(H)$ satisfying
\begin{equation*}
  \Upsilon\in \Cs(\IT,\LB(H)) 
  \cap 
  \cC(\IT_0,\LB(\dot H^{\beta-1},\dot H^{1-\beta}))
\end{equation*} 
for $\beta\in(0,1]$ as fixed in Assumption~\ref{ass:SPDE}\ref{ass:SPDE_B_beta} and
\begin{equation*}
  \sup_{t\in\IT}
  \big\|
    \Upsilon(t) 
  \big\|_{\LB(H)}
 +
  \sup_{t\in\IT_0}
  t^{1-\beta}
  \big\|
    A^{\frac{1-\beta}2}\Upsilon(t) A^{\frac{1-\beta}2}
  \big\|_{\LB(H)}
  <\infty.
\end{equation*}
On this space we introduce the family $(\triple\cdot\triple_\sigma)_{\sigma\in \R}$ of equivalent norms given by
\begin{equation*}
  \triple
    \Upsilon
  \triple_\sigma
  =
  \sup_{t\in\IT}
    e^{-\sigma t}
    \big\|
      \Upsilon(t)
    \big\|_{\LB(H)}
    +
  \sup_{t\in\IT_0}
    t^{1-\beta}
    e^{-\sigma t}
    \big\|
      A^{\frac{1-\beta}2}
      \Upsilon(t)
      A^{\frac{1-\beta}2}
    \big\|_{\LB(H)}.
\end{equation*}
The space $(\cV, \triple\cdot\triple_\sigma)$ is a Banach space since the norm is the sum of two proper Banach norms.

An operator-valued function $L\in\cV$ is called a \emph{mild solution} to~\eqref{eq:Lyapunov_var} if it satisfies \eqref{eq:Lyapunov_mild} for all $t\in\IT$ and $\phi\in H$. Existence, uniqueness and regularity of a mild solution to~\eqref{eq:Lyapunov_mild} are stated in Theorem~\ref{thm:Lyapunov} below and the equivalence of solutions to~\eqref{eq:Lyapunov_var} and~\eqref{eq:Lyapunov_mild} in Theorem~\ref{thm:mild_is_weak}. Surprisingly, the results seem to be new in our context. Since the proofs are based on standard techniques such as the Banach fixed point theorem and the Gronwall lemma, we omit them here but include them for completeness in Appendix~\ref{sec:eu} and~\ref{sec:reg}. 

\begin{theorem}\label{thm:Lyapunov}
There exists a unique mild solution $L\in\cV$ to \eqref{eq:Lyapunov_mild} that satisfies $L(\IT)\subset \Sigma^+(H)$. Moreover, the solution satisfies the following regularity estimates:
\begin{enumerate}
 \item \label{eq:spat_reg} For all $\theta_1,\theta_2\in[0,2)$ with 
$\theta_1+\theta_2<2$,  
$
  L(\IT_0)\subset\LB(\dot H^{-\theta_2}, \dot H^{\theta_1})
$ and there exists a constant $C>0$ such that for all $t \in \IT_0$
\begin{equation*}
  \|L(t)\|_{\LB(\dot H^{-\theta_2}, \dot H^{\theta_1})}
  =
  \big\|
    A^{\frac{\theta_1}2}
    L(t)
    A^{\frac{\theta_2}2}
  \big\|_{\LB(H)}
  \leq
  C
  t^{-\frac{\theta_1+\theta_2}2}.
\end{equation*}
\item \label{eq:temp_reg} For all $\theta_1,\theta_2\in[0,2)$, $\xi\in[0,1)$ with $\theta_1+\theta_2+2\xi<2$, there exists a constant $C>0$ such that for all $0 < t_1 \le t_2$
\begin{align*}
 \|L(t_2) - L(t_1)\|_{\LB(\dot H^{-\theta_2}, \dot H^{\theta_1})}
  & =
  \big\|
    A^{\frac{\theta_1}2}
    (L(t_2)-L(t_1))
    A^{\frac{\theta_2}2}
  \big\|_{\LB(H)}\\
  & \leq 
  C t_1^{-\frac{\theta_1+\theta_2+2\xi}2}|t_2-t_1|^{\xi}.
\end{align*}
\end{enumerate}
\end{theorem}

\begin{theorem}\label{thm:mild_is_weak}
Let $L\in\cV$ satisfy $L_0=G^*G$. Then, $L$ satisfies \eqref{eq:Lyapunov_mild} if and only if it satisfies the variational form~\eqref{eq:Lyapunov_var} of the Lyapunov equation for all test functions $\phi,\psi\in \dot H^2$
and in that case \eqref{eq:Lyapunov_var} is valid for all $\phi,\psi\in\dot H^\varepsilon$, $\varepsilon>0$.
\end{theorem}

\subsection{Semidiscrete approximations in space}\label{subsec:semidiscreteApprox}

Let us consider semidiscrete approximations of the Lyapunov equation~\eqref{eq:Lyapunov_var} and the SPDE~\eqref{eq:SPDE} in this subsection. For this purpose we use the approximation spaces $(V_h)_{h\in(0,1]}$ introduced in Section~\ref{sec:setting} with related operators. Let $\cV_h$ be the space 
$\cV_h = \cC(\IT,\LB(V_h))$ endowed with the norm

\begin{equation*}
  \sup_{t\in\IT}
  \big\|
    \Upsilon_h(t) 
  \big\|_{\LB(V_h)}
 +
  \sup_{t\in\IT_0}
  t^{1-\beta}
  \big\|
    A_h^{\frac{1-\beta}2}\Upsilon_h(t) A_h^{\frac{1-\beta}2}
  \big\|_{\LB(V_h)}.
\end{equation*}
The \emph{semidiscrete Lyapunov equation} reads in variational form:
Given $L_h(0) = P_h G^\ast G P_h$, find $L_h\in \cV_h$  such that for all $\phi_h,\psi_h \in V_h$
\begin{align}\label{eq:Lyapunov_semidiscrete_var}
\begin{split}
& \frac{\dd}{\dd t}
  \langle
    L_h(t) \phi_h,\psi_h
  \rangle
  +
  a(L_h(t)\phi_h,\psi_h)
  +
  a(L_h(t)\psi_h,\phi_h)
  \\
& \qquad
  =
  \langle
    R\phi_h,R\psi_h
  \rangle
  +
  \langle
    L_h(t)P_h B\phi_h,B\psi_h
  \rangle_{\LB_2^0}.
\end{split}
\end{align}
The mild formulation related to~\eqref{eq:Lyapunov_semidiscrete_var} is given for all $t \in \IT$ and $\phi_h \in V_h$ by
	\begin{align}
	\label{eq:Lyapunov_semidiscrete_mild}
	\begin{split}
	L_{h}(t) \phi 
	&=
	S_{h}(t)G^*G S_h(t) \phi_h\\
	&\quad
	+
	\int_0^t
	S_{h}(t-s)
	\big(
	R^*R 
	+
	B^*L_h(s) P_h B
	\big)
	S_h(t-s) \phi_h
	\diff s.
	\end{split}
	\end{align}
Existence and uniqueness of a solution to both equations follow from Theorem~\ref{thm:Lyapunov} and Theorem~\ref{thm:mild_is_weak} applied to~$\cV_h$. In the next proposition, we show that the regularity bounds in Section~\ref{subsec:exLyapunovSol} are uniform in~$h$ and convergence of the approximation~\eqref{eq:Lyapunov_semidiscrete_mild}.
	
\begin{prop}\label{prop:spat}
	Let $(L_h)_{h\in(0,1]} \subset \cV_h$ be the family of unique mild solutions to~\eqref{eq:Lyapunov_semidiscrete_mild}. 
 \begin{enumerate}
  \item \label{thm:spat:bound}For all $\theta_1,\theta_2\in[0,2)$ with $\theta_1+\theta_2<2$,
	there exists a constant $C>0$ such that for all $h\in(0,1)$
	\begin{equation*}
		\|A_h^{\frac{\theta_1}2}L_{h}(t)A_h^{\frac{\theta_2}2}\|_{\LB(H)}
		\leq
		Ct^{-\frac{\theta_1+\theta_2}2}.
		\end{equation*}
  \item \label{thm:spat:temp_reg} For all $\theta_1,\theta_2\in [0,2)$, $\xi\in[0,1)$ with $\theta_1+\theta_2+2\xi<2$, there exists a constant $C>0$ such that for all $h\in(0,1)$ and $0<t_1 \le t_2$
	\begin{equation*}
	\big\|
	A_h^{\frac{\theta_1}2}
	(L_h(t_2)-L_h(t_1))
	A_h^{\frac{\theta_2}2}
	\big\|_{\LB(H)}
	\leq 
	C t_1^{-\frac{\theta_1+\theta_2 + 2\xi}2}|t_2-t_1|^{\xi}.
	\end{equation*}
  \item \label{thm:spat:conv} For all $\theta_1,\theta_2\in[0,1]$, $\rho\in(0,\beta)$ with $\theta_1+\theta_2 + 2\rho<2$, there exists a constant $C>0$ such that for $h \in (0,1]$, $t\in\IT_0$
	\begin{align*}
        \big\|
            L_{h}(t)P_h-L(t)
        \big\|_{\LB(\dot H^{-\theta_2},\dot H^{\theta_1})}
        & =
		\big\|
		A^{\frac{\theta_1}2}
		\big(
		L_{h}(t)P_h-L(t) 
		\big)
		A^{\frac{\theta_2}2}
		\big\|_{\LB(H)}\\
&		 \leq
		C
		t^{-\frac{\theta_1+\theta_2 + 2\rho}2}
		h^{2\rho}.
		\end{align*}
 \end{enumerate}
\end{prop}

\begin{proof}
For every $h\in(0,1]$, Theorem~\ref{thm:Lyapunov} guarantees the existence of a constant $C=C_h>0$ such that \ref{thm:spat:bound} and \ref{thm:spat:temp_reg} hold. 

Uniformity in $h$ follows from the uniformity in~\eqref{eq:Analytic} and~\eqref{eq:Analytic2}. More precisely, every constant $C_\theta$ in the proof of Theorem~\ref{thm:Lyapunov} can be replaced by a corresponding constant $D_\theta$ in the semidiscrete setting. The only place where some extra care is needed is the estimate corresponding to~\eqref{eq:K0.1}. For this we observe that for arbitrary $K \in \LB^0_2$
\begin{align*}
\big\|B^* A_h^{\frac{\beta-1}{2}} K \big\|_{H} &= \sup_{\phi \in H, \| \phi\| = 1}  \big| \big\langle B^* A_h^{\frac{\beta-1}{2}} K, \phi \big\rangle\big| 
= \sup_{\phi \in H, \| \phi\| = 1}  \big| \big\langle K, A_h^{\frac{\beta-1}{2}} B \phi \big\rangle\big|_{\LB_2^0} \\ 
&\le \big\| A_h^{\frac{\beta-1}{2}} B \big\|_{\cL(H,\LB_2^0)} \|K\|_{\LB_2^0}
\end{align*} 
and similarly that $\| A_h^{(\beta-1)/{2}} B \|_{\cL(H,\LB_2^0)} \le \| B^* A_h^{(\beta-1)/{2}} \|_{\cL(\LB_2^0,H)}$. Therefore, we obtain $\| A_h^{(\beta-1)/{2}} B \|_{\cL(H,\LB_2^0)} = \| B^* A_h^{(\beta-1)/{2}} \|_{\cL(\LB_2^0,H)}$ and for any $Y \in \LB(V_h)$ by~\eqref{eq:AhA}
	\begin{align}
	\label{eq:B_Y_bound}
	\begin{split}
	\big\| B^* Y P_h B \big\|_{\LB(H)} &= \big\| B^* A_h^{\frac{\beta-	1}{2}}  A_h^{\frac{1-\beta}{2}} Y A_h^{\frac{1-\beta}{2}} A_h^{\frac{\beta-	1}{2}} B  \big\|_{\LB(H)} \\ 
	&\le \big\| B^* A_h^{\frac{\beta-1}{2}}  \big\|_{\LB(\LB_2^0,H)}^2 \big\| A_h^{\frac{1-\beta}{2}} Y A_h^{\frac{1-\beta}{2}}  \big\|_{\LB(H)} \\
	&\leq \big\|  B^* A^{\frac{\beta-	1}{2}} \big\|_{\LB(\LB_2^0,H)}^2 \big\| A^{\frac{1-\beta}{2}} A_h^{\frac{\beta-1}{2}} 
	\big\|_{\LB(H)}^2 \big\| A_h^{\frac{1-\beta}{2}} Y A_h^{\frac{1-\beta}{2}} \big\|_{\LB(H)} \\
	& \le  b^2 D_{1-\beta}^2 \big\| A_h^{\frac{1-\beta}{2}} Y A_h^{\frac{1-\beta}{2}} \big\|_{\LB(H)}.
	\end{split}
	\end{align}
This implies the uniform bound corresponding to~\eqref{eq:K0.1}.

Having shown the first two claims of the proof, we are ready to prove~\ref{thm:spat:conv}. First we rewrite $L_hP_h - L$ using \eqref{eq:Lyapunov_mild} and \eqref{eq:Lyapunov_semidiscrete_mild} to obtain for $\phi \in H$
\begin{align*}
	& L_{h}(t)P_h\phi-L(t)\phi \\
	& \qquad =  S_h(t) G^* G S_h (t)\phi - S(t) G^* G S(t)\phi \\
	&\quad\qquad+  \int^t_0 \big( S_h(t-s) R^* R S_h (t-s) - S(t-s) R^* R S(t-s) \big)\phi\diff s  \\
	&\quad\qquad+ \int^t_0 \big( S_h(t-s) B^* L_h (s) P_h B S_h (t-s) 
	- S(t-s) B^* L(s) B S(t-s) \big)\phi\diff s,
\end{align*}
which yields
\begin{align*}
	&\big\| A^{\frac{\theta_1}2}	\big( L_{h}(t)P_h-L(t) \big) A^{\frac{\theta_2}2} \big\|_{\LB(H)}\\  
	&\qquad\le \big\| A^{\frac{\theta_1}2} \big(S_h(t) G^* G S_h (t) - S(t) G^* G S(t) \big) A^{\frac{\theta_2}2} \big\|_{\LB(H)} \\
	&\qquad\quad+  \int^t_0 \big\|  A^{\frac{\theta_1}2}\big( S_h(t-s) R^* R S_h (t-s) - S(t-s) R^* R S(t-s) \big)A^{\frac{\theta_2}2} \big\|_{\LB(H)}\diff s  \\
	&\qquad\quad+ \int^t_0 \big\|   A^{\frac{\theta_1}2}\big( S_h(t-s) B^* L_h (s) P_h B S_h (t-s)\\
	&\hspace{25mm} - S(t-s) B^* L(s) B S(t-s) \big)A^{\frac{\theta_2}2}  \big\|_{\LB(H)} \diff s\\
	&\qquad=: I + J + K.
	\end{align*}
	We treat the three error terms separately. Using~\eqref{eq:S_smooth}, \eqref{eq:Analytic} and \eqref{eq:Eh} the term $I$ can be bounded by
	\begin{align*}
	I &\le \big\| A^{\frac{\theta_1}2} \big( S_h(t) G^* G E_h (t) \big) A^{\frac{\theta_2}2} \big\|_{\LB(H)} + \big\| A^{\frac{\theta_1}2} \big( E_h(t) G^* G S (t) \big) A^{\frac{\theta_2}2} \big\|_{\LB(H)} \\
	&\le g^2 D_{\theta_1} t^{-\frac{\theta_1}2}  \big\| E_h (t) A^{\frac{\theta_2}2} \big\|_{\LB(H)} + g^2 C_{\theta_2} t^{-\frac{\theta_2}2}  \big\| A^{\frac{\theta_1}2} E_h (t) \big\|_{\LB(H)} \\
	&\le g^2 D_{\theta_1} \big( D_{\theta_2}  + C_{\theta_2}  \big) h^{2 \rho } t^{ -\frac{\theta_1 + \theta_2 + 2\rho}2}
	\end{align*}
	and similarly the second term satisfies
	\begin{equation*}
	J \le D_{\theta_1} (D_{\theta_2} + C_{\theta_2}) r^2 h^{2 \rho} 
	\int^t_0 (t-s)^{- \frac{\theta_ 1 +\theta_2+2\rho}2} \diff s \lesssim h^{2 \rho} t^{1 - \frac{\theta_ 1 +\theta_2+2\rho}2}.
	\end{equation*}
	Adding, subtracting and applying the triangle inequality, we split $K$ into
	\begin{align*}
	K &\le \int^t_0 \big\| A^{\frac{\theta_1}2} \big( S_h(t-s) B^* L_h(s) P_h B E_h (t-s) \big) A^{\frac{\theta_2}2} \big\|_{\LB(H)} \diff s \\ 
	&\quad+ \int^t_0 \big\| A^{\frac{\theta_1}2} \big( E_h(t-s) P_h B^* L_h(s) P_h B S (t-s) \big) A^{\frac{\theta_2}2} \big\|_{\LB(H)} \diff s \\ 
	&\quad+ \int^t_0 \big\| A^{\frac{\theta_1}2} S(t-s) \big(  P_h B^* L_h(s) P_h B - B^* L(s) B \big) S(t-s) A^{\frac{\theta_2}2} \big\|_{\LB(H)} \diff s. 
	\end{align*}
By Assumption~\ref{ass:Ah}, its consequences and~\eqref{eq:B_Y_bound}, we bound the three terms by
	\begin{align*}
	 K
	 & \lesssim
	  (2 + D_L)h^{2 \rho} \, t^{\beta -\frac{\theta_1 + \theta_2+2\rho}2}\\
	  	& \qquad +  \int^t_0 (t-s)^{-\frac{\theta_1+\theta_2}2} \big\| A^{\frac{1-\beta}{2}} ( L_h(s)P_h - L(s))A^{\frac{1-\beta}{2}} \big\|_{\LB(H)} \diff s,
	\end{align*} 
	where we set
\begin{equation}\label{eq:frakL_0}
	D_L	= 	\sup_{s\in\IT_0, h \in (0,1]}s^{1-\beta}
	\|A_h^{(1-\beta)/2} L_h(s) A_h^{(1-\beta)/2} \|_{\LB(H)}.
\end{equation}

Collecting all estimates we obtain
	\begin{align}\label{eq:bootstrap1}
	\begin{split}
	&\big\| A^{\frac{\theta_1}2} ( L_h(t)P_h - L(t))A^{\frac{\theta_2}2} \big\|_{\LB(H)}\\
	&\,\lesssim t^{-\frac{\theta_1+\theta_2+2\rho}2} h^{2\rho} +  
	\int^t_0 (t-s)^{-\frac{\theta_1+\theta_2}2} \big\| A^{\frac{1-\beta}{2}} ( L_h(s)P_h - L(s))A^{\frac{1-\beta}{2}} \big\|_{\LB(H)} \diff s,	
	\end{split}
	\end{align}
where we bound all terms in~$t$ by the strongest singularity from~$I$.
Choosing $\theta_1 = \theta_2 = 1 - \beta$ and $\rho < \beta$ ensures that the exponent is bigger than $-1$, so that Gronwall's lemma (see, e.g., \cite{henry1981}) yields
	\begin{equation}\label{eq:bootstrap2}
	\big\| A^{\frac{1-\beta}{2}} ( L_h(t)P_h - L(t))A^{\frac{1-\beta}{2}} \big\|_{\LB(H)} \lesssim t^{\beta -\rho - 1} h^{2\rho}.
	\end{equation}
The general claim follows by a bootstrap argument using \eqref{eq:bootstrap2} in~\eqref{eq:bootstrap1} and~\eqref{eq:beta_integral}, which completes the proof.
\end{proof}

Having analyzed the convergence of the semidiscrete Lyapunov equation, let us continue with the semidiscrete SPDE.
Let $(X_h)_{h\in(0,1)}\subset \cC(\IT;L^2(\Omega;V_h))$ be the family of mild solutions on the finite-dimensional spaces~$V_h$ satisfying
\begin{equation}
\label{eq:moment}
\sup_{h\in(0,1]}
\sup_{t\in\IT}
\|X_h(t)\|_{L^2(\Omega;H)}
\lesssim \|X_0\|.
\end{equation}
and for all $h\in(0,1]$, $t \in \IT$, $\P$-a.s.
\begin{equation}\label{eq:SPDE_spatial}
  X_{h}(t)
  =
  S_h(t)X_0
  +
  \int_0^t
    S_h(t-s)BX_h(s)
  \diff W(s).
\end{equation}
Existence follows from \cite[Theorem~2.9(ii)]{AJK21}, where uniformity of \eqref{eq:moment} in~$h$ is deduced from~\eqref{eq:Analytic}. The proof of strong convergence is standard, cf., e.g., \cite[Theorem~3.10]{kruse2013}. Therefore we state the following proposition without proof.

\begin{prop}\label{prop:strong}
  Let $X$ be the mild solution to~\eqref{eq:SPDE_mild} and $(X_h)_{h\in(0,1)}$ be the family of unique mild solutions to~\eqref{eq:SPDE_spatial}. For all $\rho\in(0,\beta)$, there exists a constant $C>0$ such that for all $h\in(0,1]$ and $t\in \IT_0$
 \begin{equation*}
   \big\|
     X(t)-X_h(t)
   \big\|_{L^2(\Omega;H)}
   \leq
   C
   t^{-\frac\rho 2}
   h^{\rho}
   \|X_0\|.
 \end{equation*}
\end{prop}

\subsection{Connection between the Lyapunov equation and the SPDE}\label{subsec:LyapunovSolvesQoI}
 
We are now in place to prove~\eqref{eq:polarization_identity} in Theorem~\ref{thm:Phi} below. As a first step we show the equality in the semidiscrete setting of Section~\ref{subsec:semidiscreteApprox}. We therefore define for $h \in (0,1]$, $x \in V_h$ and $t \in \IT$
\begin{equation*}
	\Phi_h(x,t): = \E
	\Big[
	\int_0^t
	\|RX_h(s)\|^2
	\diff s
	+
	\|GX_h(t)\|^2
	\,\big|\,
	X_0=x
	\Big]
\end{equation*}
used in the following lemma.
 
\begin{lemma}\label{thm:PhiSemidiscr}
Let $(L_h)_{h\in (0,1)}$ be the family of unique mild solutions to~\eqref{eq:Lyapunov_semidiscrete_mild}. For all $h\in(0,1)$, $t\in\IT$, $x\in V_h$
\begin{equation*}
  \langle
    L_h(t) x,x
  \rangle
  = \Phi_h(x,t).
\end{equation*}
\end{lemma}

\begin{proof}
Fix $t\in\IT_0$ and let $v_h\colon \IT\times V_h\to \R$ satisfy for $x\in V_h$ that
\begin{equation*}
v_h(t,x)=\langle L_h(t)x,x \rangle.
\end{equation*}
In a first step, we observe that by~\eqref{eq:Lyapunov_semidiscrete_mild} and the definition of~$v_h$
\begin{equation*}
v_h(0,X_h(t))-v_h(t,X_h(0)) = \| G X_h(t)\|^2 - \langle L_h(t)X_h(0), X_h(0)\rangle.
\end{equation*}
The main part of the proof is based on applying the It\^{o} formula to deduce that
\begin{align}\label{eq:toprove}
\begin{split}
&v_h(0,X_h(t))-v_h(t,X_h(0))\\
&\qquad
  =
  -\int_0^t
    \|RX_h(s)\|^2
  \diff s
  +
  2
  \int_0^t
    \langle
      L_h(s)X_h(s),P_h B(X_h(s)) \diff W(s)
    \rangle.
\end{split}
\end{align}
Once this has been established, taking expectations on both sides completes the proof since the stochastic integral vanishes.

We now prove \eqref{eq:toprove}. In the following application of the It\^o formula, we use explicit expressions for the derivatives $\partial v_h/\partial t$, $\partial v_h/\partial x$, $\partial^2 v_h/\partial x^2$. From~\eqref{eq:Lyapunov_semidiscrete_var}, for $x,\phi,\psi\in V_h$ the time derivative $\partial v_h/\partial t$ satisfies
\begin{equation}\label{eq:time}
-\frac{\partial v_h}{\partial t}
  (t,x)
  =
  2
  \langle
    L_h(t)x,A_hx
  \rangle
  -
  \|Rx\|^2
  -
  \sum_{n\in \N}
  \langle
    L_h(t)P_hB(x)e_n,P_hB(x)e_n
  \rangle,
\end{equation}
where $(e_n)_{n=1}^\infty \subset U$ denotes an arbitrary orthonormal basis. By direct calculations the space derivatives $\partial v_h/\partial x$ and $\partial^2 v_h/\partial x^2$ are for $x,\phi,\psi\in V_h$ given by
\begin{equation}\label{eq:space}
\frac{\partial v_h}{\partial x}(t,x)(\phi)
  =
  2
  \langle
    L_h(t)x,\phi
  \rangle,
  \quad
  \frac{\partial^2 v_h}{\partial x^2} (t,x)(\phi,\psi)
  =
  2
  \langle
    L_h(t)\phi,\psi
  \rangle.
\end{equation}
Since $A_h\in \LB(V_h)$, the semidiscrete solution $X_h$ is a strong solution, meaning that $\P$-{a.s.}
\begin{equation*}
  X_h(t) 
  = 
  X_h(0)
  - 
  \int_0^t 
    A_h X_h(s) 
  \diff s
  +
  \int_0^t
    P_h B(X_h(s))
  \diff W(s).
\end{equation*}
Therefore we can apply the It\^o formula \cite[Theorem~2.4]{brzezniak2008} to the function $[0,t]\times V_h \ni (s,x) \mapsto v_h(t-s,x)$ to obtain
\begin{align*}
&v_h(0,X_h(t))-v_h(t,X_h(0))\\
&\qquad
  = -
    \int_0^t
    \frac{\partial v_h}{\partial s}
    (t-s,X_h(s))
  \diff s
  -
  \int_0^t
    \frac{\partial v_h}{\partial x}
    (t-s,X_h(s)) (A_h X_h(s))
  \diff s \\
  &\qquad\quad  +
  \int_0^t
  \frac{\partial v_h}{\partial x}
  (t-s,X_h(s)) (P_h B(X_h(s)) \diff W(s))
\\
&\qquad\quad
  +
  \tfrac12
  \sum_{n\in\N}
  \int_0^t
    \frac{\partial^2 v_h}{\partial x^2}
    (t-s,X_h(s))
    ((P_hBX_h(s))e_n,(P_hBX_h(s))e_n)
  \diff s.
\end{align*}
Inserting the expressions from~\eqref{eq:time} and~\eqref{eq:space} proves~\eqref{eq:toprove} by cancellations.
\end{proof}

We are finally in place to show~\eqref{eq:polarization_identity} even in the time dependent setting. Therefore we set with a slight abuse of notation
\begin{equation*}
	\Phi(x,t)=
	\E
	\Big[
	\int_0^t
	\|RX(s)\|^2
	\diff s
	+
	\|GX(t)\|^2
	\,\big|\,
	X_0=x
	\Big].
\end{equation*}
 The proof of the following theorem is based on the convergence of the semidiscrete approximations in Section~\ref{subsec:semidiscreteApprox} and the equality in~$V_h$.

\begin{theorem}\label{thm:Phi}
Let $X, L$ be the mild solutions to~\eqref{eq:SPDE_mild} and \eqref{eq:Lyapunov_mild}, respectively. Then for all $t\in \IT_0$ and $x\in H$
\begin{equation*}
  \langle
    L(t) x,x
  \rangle
  = \Phi(x,t)
\end{equation*}
and more specifically \eqref{eq:polarization_identity} is satisfied setting $t = T$.
\end{theorem}

\begin{proof}
By the triangle inequality we have that
\begin{align*}
 \left| \langle L(t)x,x \rangle
    - \Phi(x,t)\right|
    & \le \left|\langle (L(t)-L_h(t)P_h)x,x \rangle \right|
        + \left| \langle L_h(t) P_h x, x \rangle
                - \Phi_h(P_h x,t)\right|\\
        & \qquad + \left| \Phi_h(P_h x,t) - \Phi(x,t)\right|.
\end{align*}
We prove that the right hand side converges to zero as $h$ goes to~$0$. Proposition~\ref{prop:spat} with $\theta_1=\theta_2=0$ guarantees that
\begin{equation*}
  \lim_{h\downarrow0}
  \left|\langle (L(t)-L_h(t)P_h) x,x \rangle \right|
  \leq 
  \lim_{h\downarrow0}
  \|L(t) - L_h(t)P_h\|_{\LB(H)}\|x\|^2 
  = 0.  
\end{equation*}
The second term vanishes by Lemma~\ref{thm:PhiSemidiscr} since $\langle L_h(t) P_h x, x \rangle = \langle L_h(t) P_h x, P_h x \rangle$. 
The strong convergence in Proposition~\ref{prop:strong} and the uniform moment bounds~\eqref{eq:X_moment} and \eqref{eq:moment} imply in particular convergence of the quadratic functional and thus
\begin{align*}
  \lim_{h\downarrow0}
  \left| \Phi_h(P_h x,t) - \Phi(x,t)\right|
&\leq
  \lim_{h\downarrow0}
  \Bigg|
    \E \left[
    \int_0^t
      \big(
        \|RX_h(s)\|^2
        -
        \|RX(s)\|^2
      \big)
    \diff s
  \right]
  \Bigg|\\
&\qquad
  +
  \lim_{h\downarrow0}
  \Big|
    \E
      \big[
        \|GX_h(t)\|^2
        -
        \|GX(t)\|^2
      \big]
  \Big| = 0,
\end{align*}
i.e., the convergence of the last term.
\end{proof}

Using the polarization identity one can extend the result to bilinear forms. More specifically, let $Y$ be the mild solution to~\eqref{eq:SPDE_mild} with initial condition~$Y_0$, then for all $t\in \IT_0$ and $x, y\in H$
	\begin{equation*}
	\langle
	L(t) x,y
	\rangle
	=
	\E
	\Big[
	\int_0^t
	\langle 
	RX(s), R Y(s)
	\rangle
	\diff s
	+
	\langle 
	GX(t), G Y(t)
	\rangle
	\,\big|\,
	X_0=x, Y_0 = y
	\Big].
	\end{equation*}

Given the connection between the Lyapunov equation and $\Phi$, Theorem~\ref{thm:Phi} implies weak convergence with twice the strong rate of the semidiscrete scheme~\eqref{eq:SPDE_spatial} in a non-standard way. This extends the weak convergence result (in the multiplicative noise setting) of~\cite{AnderssonLarsson2016} to smooth noise with $\beta \in [1/2,1]$, albeit for a different class of test functions.

\begin{corollary}
Let $\Phi$ and $\Phi_h$ be the quadratic functionals in Theorem~\ref{thm:Phi} and Lemma~\ref{thm:PhiSemidiscr}, respectively.
Then, for all $\rho\in(0,\beta)$, there exists a constant $C>0$ such that for all $h\in(0,1)$ and $t \in \IT_0$
\begin{equation*}
  \big|
    \Phi(X_0,t)-\Phi_h(P_h X_0,t)
  \big|
  \leq
  C t^{-\rho} h^{2\rho} \| X_0\|^2.
\end{equation*}
\end{corollary}

\begin{proof}
Lemma~\ref{thm:PhiSemidiscr} and Theorem~\ref{thm:Phi} imply that
\begin{equation*}
\big|
    \Phi_h(P_h X_0,t)-\Phi(X_0,t)
  \big|
  = 
  \big|
    \big\langle 
      (L_h(t)P_h-L(t))X_0,X_0
    \big\rangle
  \big|.
\end{equation*}
The claim follows by Proposition~\ref{prop:spat}\ref{thm:spat:conv}.
\end{proof}

\section{Fully discrete approximation of the Lyapunov equation}
\label{sec:fully_discrete_Lyapunov}

This section is devoted to the stability and convergence analysis of a fully discrete scheme for the Lyapunov equation~\eqref{eq:Lyapunov_var}. It is based on the $\cL(V_h)$-valued implicit Euler approximation $(S_{h,\tau}^n)_{n\in\{0,\dots,N_\tau\}}$ of the semigroup $S$, introduced in Section~\ref{sec:setting}. Inspired by the mild solution~\eqref{eq:Lyapunov_mild},
we define the fully discrete approximation $L^n_{h,\tau}$ of $L(t_n)$, $n\in\{1,\dots,N_\tau\}$, by the discrete variation of constants formula
\begin{equation}\label{eq:Lyapunov_discrete_VC1}
L^n_{h,\tau}
=
S^n_{h,\tau}G^*GS^n_{h,\tau}
+
\tau
\sum_{j=0}^{n-1}
S^{n-j}_{h,\tau} 
\big(
R^*R
+
B^*L^j_{h,\tau} P_hB
\big)
S^{n-j}_{h,\tau}
\end{equation}
with $L_{h,\tau}^0=P_h G^*GP_h$. As recursion it reads
\begin{equation}\label{eq:Lyapunov_one_step}
L_{h,\tau}^n
=
S_{h,\tau}L_{h,\tau}^{n-1}S_{h,\tau}
+
S_{h,\tau}
\big(
\tau
R^* R 
+
\tau
B^* L_{h,\tau}^{n-1} P_h B
\big)
S_{h,\tau},
\end{equation} or equivalently
\begin{equation}
\label{eq:Lyapunov_one_step_A}
(P_h+\tau A_h)L_{h,\tau}^n(P_h+\tau A_h) =
L_{h,\tau}^{n-1}
+
\tau
P_h R^* R P_h
+
\tau
P_h B^* L_{h,\tau}^{n-1} P_hBP_h.
\end{equation} 
Note that $L^n_{h,\tau} \in \Sigma(V_h)$ for all $n \in \{1,\ldots,N_\tau\}$.

Before proving convergence, let us first show regularity of the fully discrete approximation, which is the analog result to Theorem~\ref{thm:Lyapunov} and Proposition~\ref{prop:spat}.

\begin{theorem}\label{thm:apriori1}
For all $c>0$ and  $\theta_1,\theta_2\in[0,2)$ with $\theta_1+\theta_2<2$, there exists a constant $C>0$ such that for $h \in (0,1]$, $\tau\le c h^{2(1-\beta)}$ and $n\in\{1,\dots,N_\tau\}$
\begin{equation*}
  \|A_h^{\frac{\theta_1}2}L_{h,\tau}^nA_h^{\frac{\theta_2}2}\|_{\LB(H)}
  \leq
  Ct_n^{-\frac{\theta_1 + \theta_2}{2}}.
\end{equation*}
\end{theorem}
\begin{proof}
We fix $n \in \{1,\ldots, N_\tau\}$. By multiplying~\eqref{eq:Lyapunov_discrete_VC1} with $A_h^{\theta_1/2}$ from the left and $A_h^{\theta_2/2}$ from the right we obtain 
\begin{align*}
  &\big\|
    A_h^{\frac{\theta_1}2}L^n_{h,\tau}A_h^{\frac{\theta_2}2}
  \big\|_{\LB(H)}\\
  &\quad\leq
  \big\|
    A_h^{\frac{\theta_1}2}
    S^n_{h,\tau}G^*GS^n_{h,\tau}
    A_h^{\frac{\theta_2}2}
  \big\|_{\LB(H)}
  +
  \tau
  \sum_{j=0}^{n-1}
    \big\|
      A_h^{\frac{\theta_1}2}
      S^{n-j}_{h,\tau}
      R^*R
      S^{n-j}_{h,\tau}
      A_h^{\frac{\theta_2}2}
    \big\|_{\LB(H)}\\
&\quad\quad
  +
  \tau
  \sum_{j=0}^{n-1}
    \big\|
      A_h^{\frac{\theta_1}2}
      S^{n-j}_{h,\tau}
      B^*L^j_{h,\tau} P_hB
      S^{n-j}_{h,\tau}
      A_h^{\frac{\theta_2}2}
    \big\|_{\LB(H)}\\
&\quad\leq
  g^2
  \big\|
    A_h^{\frac{\theta_1}2}
    S^n_{h,\tau}
  \big\|_{\LB(H)}
  \big\|
    S^n_{h,\tau}
    A_h^{\frac{\theta_2}2}
  \big\|_{\LB(H)}
  + r^2
  \tau
  \sum_{j=0}^{n-1}
    \big\|
      A_h^{\frac{\theta_1}2}
      S^{n-j}_{h,\tau}
    \big\|_{\LB(H)}
    \big\|
      S^{n-j}_{h,\tau}
      A_h^{\frac{\theta_2}2}
    \big\|_{\LB(H)}\\
&\quad\quad
  +
  \tau
  \sum_{j=0}^{n-1}
    \big\|
      A_h^{\frac{\theta_1}2}
      S^{n-j}_{h,\tau}
    \big\|_{\LB(H)}
    \big\|
      B^*L^j_{h,\tau} P_h B
    \big\|_{\LB(H)}
    \big\|
      S^{n-j}_{h,\tau}
      A_h^{\frac{\theta_2}2}
    \big\|_{\LB(H)}.
\end{align*} 
For the term containing $B$, we have by~\eqref{eq:B_Y_bound} that
\begin{equation*}
\|B^* L_{h,\tau}^j P_h B\|_{\LB(H)} \leq
b^2 
D_{1-\beta}^2
\|A_h^{\frac{1-\beta}2} L_{h,\tau}^j A_h^{\frac{1-\beta}2} \|_{\LB(H)}.
\end{equation*}
 For the other terms we use the fact that by~\eqref{eq:S_h_bound}, with $i=1,2$ and $j = 1, \ldots, n$
\begin{equation*}
	\big\|
	A_h^{\frac{\theta_i}2}
	S^{j}_{h,\tau}
	\big\|_{\LB(H)} = 
	\big\|
	S^{j}_{h,\tau}
	A_h^{\frac{\theta_i}2}
	\big\|_{\LB(H)}
	\le 
	D_{\theta_i} t_j^{-\frac{\theta_i}{2}}.
\end{equation*}
This then yields
\begin{align*}
&\big\|
    A_h^{\frac{\theta_1}2}L^n_{h,\tau}A_h^{\frac{\theta_2}2}
  \big\|_{\LB(H)}\\
  & \quad\quad\lesssim
  t_n^{-\frac{\theta_1 + \theta_2}{2}}
  +
  \tau
  \sum_{j=0}^{n-1}
    t_{n-j}^{-\frac{\theta_1 + \theta_2}{2}}
  +
  \tau
  \sum_{j=0}^{n-1}
    t_{n-j}^{-\frac{\theta_1 + \theta_2}{2}}
    \big\|
      A_h^{\frac{1-\beta}2}L^j_{h,\tau} 
      A_h^{\frac{1-\beta}2}
    \big\|_{\LB(H)}.
\end{align*}
For the first sum we have
$
  \tau\sum_{j=0}^{n-1}t_{n-j}^{-(\theta_1 + \theta_2)/2}
  \lesssim 
  \int_0^{t_n}t^{-(\theta_1 + \theta_2)/2}\diff t
  \lesssim 
  t_n^{1-(\theta_1 + \theta_2)/2}
$, so taking $\theta_1 = \theta_2 =1-\beta$ and using the discrete Gronwall lemma (cf.~\cite{elliott1992}) proves the claim for this special case and implies for $\theta_1, \theta_2\in[0,1)$, $(\theta_1 + \theta_2)/2 < 1$, in the above estimate that
\begin{align*}
  \big\|
    A_h^{\frac{\theta_1}2}L^n_{h,\tau}A_h^{\frac{\theta_2}2}
  \big\|_{\LB(H)}\lesssim
  t_n^{-\frac{\theta_1 + \theta_2}{2}}
  \Big(
    2
    +
    \tau
    \big\|
      A_h^{\frac{1-\beta}2}L_{h,\tau}^0 A_h^{\frac{1-\beta}2}
    \big\|_{\LB(H)}
  \Big)
  +
  \tau
  \sum_{j=1}^{n-1}
    t_{n-j}^{-\frac{\theta_1 + \theta_2}{2}}
    t_j^{\beta-1}.
\end{align*}
The proof is completed by observing that 
\begin{equation*}
  \tau
  \sum_{j=1}^{n-1}
  t_{n-j}^{-\frac{\theta_1 + \theta_2}{2}}
  t_j^{\beta-1}
  \lesssim
  \int_0^{t_n}
    t^{\beta-1}
    (t_n-t)^{-\frac{\theta_1 + \theta_2}{2}}
  \diff t
  \lesssim
  t_n^{\beta-\frac{\theta_1 + \theta_2}{2}}
\end{equation*} 
and that, by~Assumption~\ref{ass:Ah}\ref{eq:inv_ineq} and the bound $\tau\le c h^{2(1-\beta)}$,
\begin{equation*}
  \tau
  \big\|
    A_h^{\frac{1-\beta}2}L_{h,\tau}^0 A_h^{\frac{1-\beta}2}
  \big\|_{\LB(H)}
  \leq
  \tau
  c_0^2
  g^2
  h^{2\beta-2}
  \leq
  c
  c_0^2
  g^2
  \lesssim 1.\qedhere
\end{equation*}
\end{proof}

We are now in place to prove convergence of the fully discrete approximation of~\eqref{eq:Lyapunov_mild} with the same convergence rate~$2\rho$ in space as in the semidiscrete setting in Proposition~\ref{prop:spat} and rate~$\rho$ in time.

\begin{theorem}\label{thm:conv1}
For all $c>0$, $\rho\in(0,\beta)$ and $\theta\in[0,1)$ with $\rho+\theta<1$, there exists a constant $C>0$ satisfying for 
$h\in(0,1)$, $\tau\leq ch^{2 (1 - \beta)/(1-\rho)}$ and
 $n\in\{1,\dots,N_\tau\}$
\begin{equation*}
 \|L_{h,\tau}^nP_h-L(t_n) \|_{\LB(\dot{H}^{-\theta},\dot{H}^{\theta})}
=
\big\|
    A^{\frac{\theta}2}
    \big(
      L_{h,\tau}^nP_h-L(t_n) 
    \big)
    A^{\frac{\theta}2}
  \big\|_{\LB(H)}
  \leq
  C
  t_n^{-\rho-\theta}
  (h^{2\rho} + \tau^\rho).
\end{equation*}
\end{theorem}
\begin{proof}
	With Proposition~\ref{prop:spat}, the triangle inequality, \eqref{eq:AhA} and~\eqref{eq:AhA2} it suffices to prove that under the conditions of this theorem, there exists $C>0$ such that for $h\in(0,1)$, $\tau\leq ch^{2 (1 - \beta)/(1-\rho)}$ and $n\in\{1,\dots,N_\tau\}$
	\begin{equation*}
	\big\|
	A_h^{\frac{\theta}2}
	\big(
	L_{h,\tau}^nP_h-L_h(t_n) P_h
	\big)
	A_h^{\frac{\theta}2}
	\big\|_{\LB(H)}
	\leq
	C
	t_n^{-\rho-\theta} \tau^\rho.
	\end{equation*}
We introduce the right-continuous interpolation $\tilde S=(\tilde S(t))_{t\in[0,T)}$ of $S_{h,\tau}$ given by
\begin{equation*}
  \tilde S(t)
  =
  \sum_{n=1}^{N_\tau} \mathbf{1}_{[t_{n-1},t_n)}(t)S^n_{h,\tau},
\end{equation*}
for which we by~\eqref{eq:AhA} and \eqref{eq:S_h_bound} have for $r \in [0,1]$ the existence of a constant $D_r$ such that for all $t \in \IT_0$
\begin{equation}
	\label{eq:S_tilde_bound}
	\| \tilde S(t) A_h^{\frac{r}{2}}\|_{\LB(H)} = \| A_h^{\frac{r}{2}} \tilde S(t) \|_{\LB(H)}
	\leq
	D_r t^{-\frac{r}{2}}.
\end{equation}
We also introduce the corresponding error operator $\tilde E=\tilde S-S_h$ and extend the fully discrete solution to continuous time by $\tilde L(t)=L_{h,\tau}^n P_h$, for $t\in[t_n,t_{n+1})$.  From \eqref{eq:Analytic}, \eqref{eq:Analytic2} and \eqref{eq:E_h_bound} we obtain the existence of a constant $D_\theta$ such that for all $t>0$ 
\begin{equation}
\label{eq:E_bound}
  \|A_h^{\frac{\theta}2}\tilde E(t)\|_{\LB(H)}
  \leq
  D_\theta t^{-\frac{\theta}2-\rho}\tau^{\rho}
  .
\end{equation}
Using this notation it follows from \eqref{eq:Lyapunov_discrete_VC1} that for all $n\in\{1,\dots, N_\tau\}$ and $\phi \in H$
\begin{equation*}
  \tilde L(t_n)\phi
=  
   S^n_{h,\tau}G^*G S^n_{h,\tau}\phi
  +
  \int_0^{t_n}
    \tilde S(t_n-s)
    \big(
      R^*R
      +
      B^* \tilde L(s) P_h B 
    \big)
    \tilde S(t_n-s)\phi
  \diff s.
\end{equation*}
Since $A_h^{\theta/2}, S_h, \tilde S, L_{h,\tau}, L_h$ and $\tilde{E}$ are self-adjoint at all times, this equality and~\eqref{eq:Lyapunov_semidiscrete_mild} along with
\begin{equation*}
\big\|A_h^{\frac{\theta}2}
\big(
L_{h,\tau}^n-L_h(t_n)
\big)A_h^{\frac{\theta}2}\big\|_{\LB(H)}
=
\sup_{\phi \in H, \| \phi\| = 1}
	\big|
	\big\langle
	A_h^{\frac{\theta}2}
	\big(
	L_{h,\tau}^n-L_h(t_n)
	\big)A_h^{\frac{\theta}2} \phi,\phi
	\big\rangle \big|, 
\end{equation*}
yield
\begin{align*}
\big\| &
    A_h^{\frac{\theta}2}
    \big(
      L_{h,\tau}^n-L_h(t_n)
    \big)
    A_h^{\frac{\theta}2}
  \big\|_{\LB(H)}\\
&\leq
	g^2 
	\big\|
	A_h^{\frac{\theta}2}
	 \big(S^n_{h,\tau}+
	S_h(t_n) \big)
	\big\|_{\LB(H)} 
  \big\|
   E^n_{h,\tau}
    A_h^{\frac{\theta}2}
  \big\|_{\LB(H)}\\
&\quad
  +
  r^2 
  \int_0^{t_n}
  \big\|
  A_h^{\frac{\theta}2}
  \big(\tilde S(t_n-s)+
  S_h(t_n-s)\big)
  \big\|_{\LB(H)} 
  \big\|
  \tilde E(t_n-s)
  A_h^{\frac{\theta}2}
  \big\|_{\LB(H)}
  \diff s\\
&\quad
  +
  \int_0^{t_n}
    \big\|
    A_h^{\frac{\theta}2}
    \big(\tilde S(t_n-s)+
    S_h(t_n-s)\big)
    \big\|_{\LB(H)} 
    \big\|
    B^* L_h(s) P_h B
    \tilde E(t_n-s)
    A_h^{\frac{\theta}2}
    \big\|_{\LB(H)}
  \diff s\\
&\quad
  +
  \int_0^{t_n}
    \big\|
      A_h^{\frac{\theta}2}
      \tilde S(t_n-s)B^* (\tilde L(s)-L_h(s)) P_h B \tilde S(t_n-s)
      A_h^{\frac{\theta}2}
    \big\|_{\LB(H)}
  \diff s =:
  \sum_{i=1}^4 I_i^n.
\end{align*}
For the first term we obtain with~\eqref{eq:E_h_bound}, \eqref{eq:S_h_bound} and~\eqref{eq:Analytic} that
\begin{equation*}
  I_1^n
\leq
  g^2
  \big\|
    A_h^{\frac{\theta}2}
    E^n_{h,\tau}
  \big\|_{\LB(H)}  
  \Big(
    \big\|
      S^n_{h,\tau}
      A_h^{\frac{\theta}2}
    \big\|_{\LB(H)}
    +
    \big\|
      A_h^{\frac{\theta}2}
      S_h(t_n)
    \big\|_{\LB(H)}
  \Big)
  \leq
  2 g^2 D_\theta^2
  t_n^{-\rho-\theta}
    \tau^\rho.
\end{equation*}
Similarly we use \eqref{eq:S_tilde_bound}, \eqref{eq:S_h_bound} and \eqref{eq:E_bound} for the next term to see that
\begin{align*}
  I_2^n
&\leq
  2 r^2
  D_\theta^2
  \Bigg(
    \int_0^{t_n}
      (t_n-s)^{-\theta-\rho}
    \diff s
  \Bigg)\tau^\rho
  \lesssim
  \tau^\rho.
\end{align*}
Using \eqref{eq:S_tilde_bound}, \eqref{eq:Analytic}, \eqref{eq:B_Y_bound}, \eqref{eq:frakL_0}, \eqref{eq:E_bound} and \eqref{eq:beta_integral}  yields 
\begin{align*}
  I_3^n 
&\leq
  2
  b^2
  D_{1-\beta}^2
  D_L
  D_\theta^2
  \Bigg(
    \int_0^{t_n}
      s^{\beta-1}(t_n-s)^{-\rho-\theta}
    \diff s
  \Bigg)
    \tau^\rho
  \lesssim
  t_n^{\beta-\rho-\theta}
  \tau^\rho.
\end{align*}
For the last term we add and subtract a piecewise constant approximation of~$L_h$, $\tilde L_h(t)=L_h(t_n)$ for $t\in[t_n,t_{n+1})$. With~\eqref{eq:B_Y_bound} we obtain
\begin{align*}
  I_4^n
&\leq
  b^2 D_{1-\beta}^2 \int_0^{t_n}
    \big\|
      A_h^{\frac{\theta}2}
      \tilde S(t_n-s)
    \big\|_{\LB(H)}^2 \Big(
    \big\|
    A_h^{\frac{1-\beta}2}
    \big(\tilde L_h(s)-L_h(s)\big)
    A_h^{\frac{1-\beta}2}
    \big\|_{\LB(H)}\\
    &\hspace{60mm}
    +
    \big\|
    A_h^{\frac{1-\beta}2}
    \big(\tilde L(s)-\tilde L_h(s)\big)
    A_h^{\frac{1-\beta}2}
    \big\|_{\LB(H)}
    \Big)
  \diff s.
\end{align*}
Proposition~\ref{prop:spat}\ref{thm:spat:temp_reg} yields the existence of a constant $D_{\beta,\rho}$ so that for all $s \in [t_j, t_{j+1})$ and $j \in \{1, \ldots, n-1\}$ the first of the two terms in the sum is bounded by
\begin{equation*}
	\big\|
	A_h^{\frac{1-\beta}2}
	\big(L_h(t_j)-L_h(s)\big)
	A_h^{\frac{1-\beta}2}
	\big\|_{\LB(H)}
	\le D_{\beta,\rho} t_j^{\beta-\rho-1} |t_j-s|^\rho\le D_{\beta,\rho} t_j^{\beta-\rho-1} \tau^\rho.
\end{equation*}
For $s\in [0, \tau)$, however, we use~\eqref{eq:frakL_0} and Assumption~\ref{ass:Ah}\ref{eq:inv_ineq} to bound it by
\begin{equation*}
	\big\|
	A_h^{\frac{1-\beta}2}G^*G
	A_h^{\frac{1-\beta}2}
	\big\|_{\LB(H)} 
	+
	\big\|
	A_h^{\frac{1-\beta}2}
	L_h(s)A_h^{\frac{1-\beta}2}
	\big\|_{\LB(H)}
	\le g^2 D_{\beta-1}^2 h^{2\beta-2} + D_L s^{\beta-1}.
\end{equation*}
Noting also that by~\eqref{eq:S_h_bound}, $\|
A_h^{\theta/2}
\tilde S(t_n-s)
\|_{\LB(H)} \le D_\theta t_{n-j}^{-\theta/2}$ for $s\in [t_j,t_{j+1})$ and $j \in \{0,\ldots,n-1\}$ and that $\tilde L_h(s) - \tilde L(s) = 0$ for $s \in [0,\tau_1)$, we find that
\begin{align*}
	I^n_4 &\lesssim t_n^{-\theta} \Big(\int^\tau_0 s^{\beta-1} \diff s + \tau h^{2\beta-2}\Big) + \tau^{1+\rho} \sum_{j = 1}^{n-1} t_{n-j}^{-\theta} t_j^{\beta-\rho-1} \\
	&\quad\quad + \tau \sum_{j = 1}^{n-1} t_{n-j}^{-\theta} \big\|
	A_h^{\frac{1-\beta}2}
	\big(
	L_{h,\tau}^j-L_h(t_j)
	\big)
	A_h^{\frac{1-\beta}2}
	\big\|_{\LB(H)} \\
	&\lesssim t_n^{-\theta} \big(\tau^\beta + \tau^\rho \big) + t_n^{\beta-\rho-\theta} \tau^\rho+ \tau \sum_{j = 1}^{n-1} t_{n-j}^{-\theta} \big\|
	A_h^{\frac{1-\beta}2}
	\big(
	L_{h,\tau}^j-L_h(t_j)
	\big)
	A_h^{\frac{1-\beta}2}
	\big\|_{\LB(H)},
\end{align*}
where the last inequality follows by
$
\tau
\sum_{j=1}^{n-1}
t_{n-j}^{-\theta}
t_j^{\beta-1-\rho}
\lesssim
\int_0^{t_n}
(t_{n}-s)^{-\theta}
s^{\beta-1-\rho}    
\diff s
\lesssim
t_n^{\beta-\rho-\theta}
$
and the fact that the coupling $\tau\leq ch^{2 (1 - \beta)/(1-\rho)}$ yields $\tau^{\rho}\tau^{1-\rho} h^{2 \beta-2} \leq c \tau^{\rho}$. 

Collecting the estimates, the sum of all four terms is bounded by
\begin{equation*}
  \sum_{i=1}^4 I_i^n
   \lesssim
  t_n^{-\theta-\rho}\tau^\rho
  +
  \tau\sum_{j=1}^{n-1}
    t_{n-j}^{-\theta}
    \big\|
      A_h^{\frac{1-\beta}2}
      \big(
        L_{h,\tau}^j-L_h(t_j)
      \big)
      A_h^{\frac{1-\beta}2}
    \big\|_{\LB(H)}.
\end{equation*}
The choice $\theta = 1-\beta$ implies with the discrete Gronwall lemma that
\begin{equation*}
\big\|
    A_h^{\frac{1-\beta}2}
    \big(
      L_{h,\tau}^n-L_h(t_n)
    \big)
    A_h^{\frac{1-\beta}2}
  \big\|_{\LB(H)}
  \lesssim
  t_n^{\beta-\rho-1}\tau^\rho,
\end{equation*}
which shows the claim for this special case.
Similarly to Proposition~\ref{prop:spat}, the proof is completed by a bootstrap argument.
\end{proof}

As a consequence, we obtain convergence of the approximation of the quadratic functional~\eqref{eq:Phi} by the Lyapunov equation of up to rate $\beta$ in time and double this rate in space.
\begin{corollary}
	\label{cor:Phi_bound_Lyap}
For all $c>0$ and $\rho\in(0,\beta)$, there exists a constant $C>0$ satisfying for $h\in(0,1)$, $\tau\leq ch^{2 (1 - \beta)/(1-\rho)}$ and $x \in H$
\begin{equation*}
\big|
\Phi(x) - \Phi_{h,\tau}^\mathrm{L}(x)
\big|
\leq 
C T^{-\rho}
(h^{2\rho} + \tau^\rho) \|x\|^2,
\end{equation*}
where $\Phi_{h,\tau}^\mathrm{L}(x) = \langle L_{h,\tau}^N P_h x, P_h x\rangle$.
\end{corollary}
\begin{proof}
	Using Theorems~\ref{thm:Phi} and~\ref{thm:conv1} with $x \in \dot{H}^{0}$ we directly obtain
	\begin{align*}
\big|
\Phi(x) - \Phi_{h,\tau}^\mathrm{L}(x)
\big| &= |\langle (L(T) - L_{h,\tau}^{N_\tau} P_h) x, x \rangle| \\ 
&\le \big\|
	L_{h,\tau}^{N_\tau} P_h-L(T) 
	\big\|_{\LB(H)}
	\|x \|^2 \lesssim T^{-\rho}
	(h^{2\rho} + \tau^\rho) \| x \|^2.\qedhere	
	\end{align*}
\end{proof}

\section{Fully discrete SPDE approximation}
\label{sec:fully_discrete_spde}

In Section~\ref{sec:Lyapunov_equation_and_SPDE} we have shown the connection~\eqref{eq:polarization_identity} between the Lyapunov equation and the SPDE by the analogous equality of the space approximations of the equations. In this section we prove that a similar relation holds in the fully discrete setting up to a sufficiently fast converging error. This connection implies weak convergence of a fully discrete approximation of the SPDE~\eqref{eq:SPDE}.

The fully discrete approximation of~\eqref{eq:SPDE} is obtained by a semiimplicit Euler--Maruyama scheme. Let $(X_{h,\tau})_{h,\tau\in(0,1)}$ be the family of discrete stochastic processes
satisfying $X_{h,\tau}^0 = P_h X_0$ and for all $n\in\{1,\dots,N_\tau\}$ $\P$-a.s
\begin{equation}\label{eq:full_disc_SPDE}
  X_{h,\tau}^n
  +
  \tau A_h X_{h,\tau}^n
  =
  X_{h,\tau}^{n-1} + 
  B(X_{h,\tau}^{n-1}) \Delta W^{n-1},
\end{equation}
where $\Delta W^n = W(t_{n+1}) - W(t_n)$ denotes the increment of the Wiener process. Using the fact that $S_{h,\tau} = (P_h+\tau A_h)^{-1}$ the recursion can be rewritten as
\begin{equation}\label{eq:full_disc_SPDE2}
  X_{h,\tau}^n
  =
  S_{h,\tau}
  X_{h,\tau}^{n-1}
  +
  S_{h,\tau}B(X_{h,\tau}^{n-1}) \Delta W^{n-1},
\end{equation}
which leads to the discrete variation of constants formula
\begin{equation}\label{eq:SPDE_discrete_VC}
  X_{h,\tau}^n
  =
  S_{h,\tau}^n  X_0
  +
  \sum_{j=0}^{n-1}
    S_{h,\tau}^{n-j}
    B(X_{h,\tau}^j)
  \Delta W^j.
\end{equation}

An induction shows that $X_{h,\tau}^n\in L^p(\Omega;H)$ for all $h,\tau\in(0,1)$, $n\in\{1,\dots,N_\tau\}$, $p\in[2,\infty)$ and by a classical Gronwall argument one obtains for all $p \ge 2$ and $0\leq r < \beta$ the existence of a constant~$D_{p,r}$ such that

\begin{equation}\label{eq:moment_full}
\sup_{h,\tau\in(0,1)}
\sup_{n\in\{1,\dots,N_{\tau}\}}
t_n^{\frac{r}{2}}
\|A_h^{\frac{r}2} X_{h,\tau}^n\|_{L^p(\Omega;H)}
\leq
D_{p,r} \|X_0\|.
\end{equation}
We omit the details and refer to \cite[Proposition~3.16]{AnderssonKruseLarsson} for a proof of a similar stability result. Apart from this, we use the following lemma in our weak convergence proof.
\begin{lemma}
	\label{lemma:K_bound}
Let $K \in \cL(H)$. For all $\theta < \beta$, there exists a constant $C>0$ such that for all $n \in \{1, \ldots,N_\tau\}$ and $h \in (0,1]$
	\begin{equation*}
	\big|\E\big[\langle (S_{h,\tau}-P_h) K X_{h,\tau}^n, X_{h,\tau}^n \rangle\big]\big| \le C
	t_n^{-\theta} \tau^\theta \|K\|_{\cL(H)}\| X_0 \|^2.
	\end{equation*}	
\end{lemma}

\begin{proof}
	We use the It\^o isometry and the fact that the centered increment $\Delta W^{j}$ is independent of $X_{h,\tau}^{j}$ in~\eqref{eq:SPDE_discrete_VC} to find that
	\begin{align*}
	&\E\big[\langle (S_{h,\tau}-P_h)K X_{h,\tau}^n, X_{h,\tau}^n \rangle\big] \\ 
	&\qquad= 
	\E\big[\langle  (S_{h,\tau}-P_h)K S_{h,\tau}^n X_0, S_{h,\tau}^n X_0 \rangle\big] \\
	&\qquad\quad+
	\tau
	\sum_{j=0}^{n-1}
	\E\big[\langle  (S_{h,\tau}-P_h)K
	S_{h,\tau}^{n-j}
	B(X_{h,\tau}^j),S_{h,\tau}^{n-j}
	B(X_{h,\tau}^j) \rangle_{\cL^0_2} \big].
	\end{align*}
In a first step, this gives that
	\begin{align*}
	&\big|\E\big[\langle (S_{h,\tau}-P_h)K X_{h,\tau}^n, X_{h,\tau}^n \rangle\big]\big| \\ 
	&\quad\le
	\big| \langle A_h^{-\theta} (S_{h,\tau}-P_h) K S_{h,\tau}^n X_0, A_h^{\theta} S_{h,\tau}^n X_0 \rangle\big| \\
	&\quad\quad+
	\tau \sum_{j=0}^{n-1}
	\big|\E\big[\langle  A_h^{-\theta} (S_{h,\tau}-P_h)K
	A_h^{\frac{1-\beta}{2}} S_{h,\tau}^{n-j} A_h^{\frac{\beta-1}{2}}
	B(X_{h,\tau}^j), \\
	&\quad\hspace{2.25cm}A_h^{\frac{1-\beta+2\theta}{2}} S_{h,\tau}^{n-j} A_h^{\frac{\beta-1}{2}}
	B(X_{h,\tau}^j) \rangle_{\cL^0_2} \big]\big| \\
	&\quad\le  \|A_h^{-\theta} (S_{h,\tau}-P_h)\|_{\cL(H)} \|K\|_{\cL(H)}\\
	&\quad\quad\times
	\Big( \|S_{h,\tau}^n\|_{\cL(H)} \|A_h^{\theta} S_{h,\tau}^n\|_{\cL(H)} \|X_0\|^2 \\
	&\quad\quad\quad\quad+
	 \tau \sum_{j=0}^{n-1}
	 \|A_h^{\frac{1-\beta}{2}} S_{h,\tau}^{n-j}\|_{\cL(H)} \|A_h^{\frac{\beta-1}{2}}B\|_{\cL_2^0}^2 \|A_h^{\frac{1-\beta+2\theta}{2}} S_{h,\tau}^{n-j}\|_{\cL(H)}
	 \E\big[\|X_{h,\tau}^n\|^2\big] \Big).
	\end{align*} 
	Using \eqref{eq:S_h_Holder}, \eqref{eq:S_h_bound}, \eqref{eq:AhA2} and \eqref{eq:moment_full} completes the proof.
\end{proof}
 
We are now in place to state the fully discrete version of~\eqref{eq:polarization_identity}, which implies weak convergence stated in Corollary~\ref{cor:weak_full} below. We set
\begin{equation}\label{eq:Phi_full}
\Phi_{h,\tau}(x,n)
=
\E
\Bigg[
\big\|
G X_{h,\tau}^n
\big\|^2
+
\tau
\sum_{k=1}^{n}
\big\|
R X_{h,\tau}^{n-k}
\big\|^2
\,\Big|\,
X_{h,\tau}^{0} = P_h x
\Bigg]
\end{equation}
for $n=0,\ldots,N_\tau$.

\begin{theorem}\label{thm:disc_rep}
	Let $\Phi_{h,\tau}$ be the functional given by \eqref{eq:Phi_full} and let $L^n_{h,\tau}$ be given by~\eqref{eq:Lyapunov_discrete_VC1}. For all $c>0$ and $\rho\in(0,\beta)$, there exists a constant $C>0$ satisfying for $h\in(0,1)$, $\tau\leq ch^{2 (1 - \beta)/(1-\rho)}$, $n\in\{1,\dots,N_\tau\}$ and $x \in H$
	\begin{equation*}
	\big| \big\langle
	L_{h,\tau}^n P_h x , P_h x
	\big\rangle - \Phi_{h,\tau}(x,n)
	\big|
	\leq
	C
	\tau^\rho
	\|x\|^2.
	\end{equation*}
\end{theorem}

\begin{proof}
	By a telescoping sum argument and the fact that $L_{h,\tau}^0 = P_hG^*GP_h$ we obtain
	\begin{align*}
	&\big\langle
	L_{h,\tau}^n P_h x , P_h x
	\big\rangle\\  
	&\qquad
	=
	\big\langle
	L_{h,\tau}^0 X_{h,\tau}^n , X_{h,\tau}^n
	\big\rangle
	+
	\sum_{k=1}^{n} 
	\big\langle
	L_{h,\tau}^{k} X_{h,\tau}^{n-k},X_{h,\tau}^{n-k}
	\big\rangle
	-
	\big\langle
	L_{h,\tau}^{k-1} X_{h,\tau}^{n-k+1},X_{h,\tau}^{n-k+1}
	\big\rangle
	\\  
	&\qquad
	=
	\E
	\Bigg[
	\big\|
	GX_{h,\tau}^n
	\big\|^2
	+
	\sum_{k=1}^{n}
	\big\langle
	L_{h,\tau}^{k} X_{h,\tau}^{n-k},X_{h,\tau}^{n-k}
	\big\rangle
	-
	\big\langle
	L_{h,\tau}^{k-1} X_{h,\tau}^{n-k+1},X_{h,\tau}^{n-k+1}
	\big\rangle
	\Bigg]
	\end{align*}
	so that 
	\begin{align}\allowdisplaybreaks
	\label{eq:weak_proof_term_1}
		\begin{split}
		&\big\langle
		L_{h,\tau}^n P_h x , P_h x
		\big\rangle - \Phi_{h,\tau}(x,n) \\ 
		&\quad= \sum_{k=1}^{n} \E
		\big[ 	
		\big\langle
		L_{h,\tau}^{k} X_{h,\tau}^{n-k},X_{h,\tau}^{n-k}
		\big\rangle
		-
		\big\langle
		L_{h,\tau}^{k-1} X_{h,\tau}^{n-k+1},X_{h,\tau}^{n-k+1}
		\big\rangle
		-\tau \big\|
		RX_{h,\tau}^{n-k} \big\|^2 \big]\\
	&\quad=  \sum_{k=1}^{n}\tau\E
\big[
\big\langle \big(S_{h,\tau} B^*  L_{h,\tau}^{k-1} P_h B S_{h,\tau} - B^* S_{h,\tau} L_{h,\tau}^{k-1} S_{h,\tau} B \big)X_{h,\tau}^{n-k}, X_{h,\tau}^{n-k} \big\rangle \big] \\
&\hspace*{10em}+
\tau \E
\big[\big\langle R (P_h+S_{h,\tau})  X_{h,\tau}^{n-k}, R (S_{h,\tau}-P_h)  X_{h,\tau}^{n-k} \big\rangle\big]\\
		&\quad=:\sum_{k=1}^{n} I^{n,1}_k +I^{n,2}_k.
		\end{split}
	\end{align}
The second equality follows from the identity
	\begin{align*}
	&\E\big[ 
		\big\langle
		L_{h,\tau}^{k-1} X_{h,\tau}^{n-k+1},X_{h,\tau}^{n-k+1}
		\big\rangle \big] \\
		&\quad= \E\big[ 	
		\big\langle S_{h,\tau} L_{h,\tau}^{k-1} S_{h,\tau} X_{h,\tau}^{n-k}, X_{h,\tau}^{n-k} \big\rangle \big] + \tau \E\big[\big\langle 
		B^* S_{h,\tau}L_{h,\tau}^{k-1} S_{h,\tau} B X_{h,\tau}^{n-k},
		 X_{h,\tau}^{n-k}
		\big\rangle \big],
	\end{align*}
obtained by~\eqref{eq:full_disc_SPDE}, the It\^o isometry and the independence of $\Delta W^{n-k}$ and $X_{h,\tau}^{n-k}$, along with \eqref{eq:Lyapunov_one_step} and that $\| u\|^2 - \|v\|^2 = \langle u+v,u-v\rangle$.
	
For the first term $I^{n,1}_k$, since $L_{h,\tau}^0=P_hG^*GP_h$, by the triangle inequality, \eqref{eq:S_h_bound}, \eqref{eq:AhA2},  Assumption~\ref{ass:Ah}\ref{eq:inv_ineq} and \eqref{eq:moment_full} we obtain for $k=1$ 
	\begin{align*}
		|I^{n,1}_1| &= \tau\big|\E
		\big[
		\big\langle \big( S_{h,\tau} B^* A_h^{\frac{\beta-1}{2}}  A_h^{\frac{1-\beta}{2}}   G^*G A_h^{\frac{1-\beta}{2}} A_h^{\frac{\beta-1}{2}}B S_{h,\tau} \\
		&\hspace{1.5cm}-  B^* A_h^{\frac{\beta-1}{2}} A_h^{\frac{1-\beta}{2}} S_{h,\tau} G^*G S_{h,\tau} A_h^{\frac{1-\beta}{2}} A_h^{\frac{\beta-1}{2}} B \big)X_{h,\tau}^{n-1}, X_{h,\tau}^{n-1} \big\rangle \big]\big| \\
		&\le 2 \tau g^2  \|S_{h,\tau}\|^2 \|A_h^{\frac{\beta-1}{2}} B\|^2  \|A_h^{\frac{1-\beta}{2}}\|^2 \|X_{h,\tau}^{n-1}\|^2 \\
		&\le 2 \tau h^{2(\beta -1)} g^2 D_0^2 D_{\beta-1}^2 b^2 D_{1-\beta}^2   D^2_{2,0} \| x \|^2 
		\lesssim \tau h^{2(\beta -1)} \| x \|^ 2.
	\end{align*}
	In the case that $k = 2, \ldots, n$, we use the fact that $L_{h,\tau}^{k-1} \in \Sigma(V_h)$ to obtain the split
	\begin{align*}
		I^{n,1}_k &= 
		\tau \E \big[ \big\langle (S_{h,\tau} - P_h ) B^* L_{h,\tau}^{k-1} P_h  B (S_{h,\tau} + P_h) X_{h,\tau}^{n-k}, X_{h,\tau}^{n-k} \big\rangle \big] \\
		 &\quad- 
		 \tau \E \big[ \big\langle B^*(S_{h,\tau} - P_h )  L_{h,\tau}^{k-1}  (S_{h,\tau} + P_h) B X_{h,\tau}^{n-k}, X_{h,\tau}^{n-k} \big\rangle \big] =: J^{n,1}_k - J^{n,2}_k.
	\end{align*}
	
	The term $|J^{n,1}_k|$ is for $k \neq n$ handled by Lemma~\ref{lemma:K_bound} with
	\begin{equation*}
		K =  B^* L_{h,\tau}^{k-1} B (S_{h,\tau} + P_h)  
	\end{equation*}
	which is by~\eqref{eq:AhA2}, Theorem~\ref{thm:apriori1} and~\eqref{eq:S_h_bound} bounded by
	\begin{align*}
	\| K \|_{\cL(H)} 
	&\le  D_{\beta-1}^2 b^2 D^L_{1-\beta,1-\beta} t_{k-1}^{\beta-1} (1 + D_0),
	\end{align*}
	where $D^L_{\theta_1,\theta_2}$, $\theta_1, \theta_2 \in [0,2)$, denotes the constant~$C$ in~Theorem~\ref{thm:apriori1}. For $k=n$ we note that~\eqref{eq:S_h_Holder} implies 
	\begin{equation*}
	\label{eq:K_bound_var}
	\big|\langle (S_{h,\tau}-P_h) K x, x \rangle\big| \le \|S_{h,\tau}-P_h\|_{\cL(H)} \|K\|_{\cL(H)} \|x\|^2 \le D_0 \|K\|_{\cL(H)}\| x \|^2.
	\end{equation*}
	The term $|I^{n,2}_k|$ is treated similarly to $|J^{n,1}_k|$, using for $K=R^* R (P_h+S_{h,\tau})$ the bound
	$
	\| K \|_{\cL(H)} \le r^2 (1 + \|S_{h,\tau}\|_{\cL(H)}) \le r^2 (1 + D_0)
	$
	in Lemma~\ref{lemma:K_bound} and~\eqref{eq:K_bound_var}. We obtain that $|J^{n,1}_k|$ is bounded by a constant times $\tau^{1+\rho} t_{k-1}^{\beta-1} t_{n-k}^{-\rho} \| x \|^ 2$ in the case that $k \neq n$, and a constant times $\tau t_{n-1}^{\beta-1}  \| x \|^ 2$ for $k = n$. Similarly, $|I^{n,2}_k|$ is bounded by a constant times $\tau^{1+\rho} t_{n-k}^{-\rho} \| x \|^ 2$ and $\tau \| x \|^ 2$ for $k \neq n$ and $k=n$, respectively.
	
	For the term $|J^{n,2}_k|$, with $k \neq n$, we obtain from \eqref{eq:AhA2}, \eqref{eq:S_h_Holder}, Theorem~\ref{thm:apriori1}, \eqref{eq:S_h_bound} and \eqref{eq:moment_full} that
	\begin{align*}
		|J^{n,2}_k| 
		&\le \tau^{1+\rho} t_{k-1}^{\beta - \rho - 1} D_{\beta-1}^2 b^2 D_{\rho} D^L_{1-\beta+2\rho,1-\beta} (D_0 + 1) D^2_{2,0}\|x\|^2 \lesssim \tau^{1+\rho} t_{k-1}^{\beta - \rho - 1} \| x \|^ 2.
	\end{align*}
	For $k = n$ one obtains the same bound without the term $D^2_{2,0}$ since \eqref{eq:moment_full} is not used.  
	Collecting the estimates, we bound \eqref{eq:weak_proof_term_1} for $n > 1$ by 
	\begin{align*}
		\big|\big\langle L_{h,\tau}^n P_h x , P_h x \big\rangle - \Phi_{h,\tau}(x,n)\big| 
		&\lesssim \Bigg( \tau h^{2(\beta-1)} + \tau^{1+\rho} t_{n-1}^{-\rho}  + \tau t_{n-1}^{\beta-1} + \tau^{1+\rho} t_{n-1}^{\beta - \rho - 1} + \tau \\
		&\hspace{10mm} + \tau^{1+\rho} \sum_{k=2}^{n-1} (t_{k-1}^{\beta-1} t_{n-k}^{-\rho} + t_{k-1}^{\beta - \rho - 1} + t_{n-k}^{-\rho}) \Bigg) \| x \|^ 2. 
	\end{align*}
	We note that $\tau\leq c h^{2(1-\beta)/(1-\rho)}$ yields $\tau h^{2 (\beta -1)} = \tau^\rho \tau^{1-\rho} h^{2 (\beta -1)} \leq c^{1-\rho} \tau^\rho$. Moreover, the identity~\eqref{eq:beta_integral} implies
	\begin{equation*}
		\tau \sum_{k=2}^{n-1} (t_{k-1}^{\beta-1} t_{n-k}^{-\rho}+ t_{k-1}^{\beta - \rho - 1} + t_{n-k}^{-\rho}) \lesssim t_{n}^{\beta-\rho}  + t_{n}^{1-\rho} \lesssim T^{\beta-\rho} + T^{1-\rho} .
	\end{equation*}
	 These facts, along with the bounds $\tau^{1+\rho} t_{n-1}^{-\rho} \le \tau$,  $\tau t_{n-1}^{\beta-1} = \tau^\beta \tau^{1-\beta} t_{n-1}^{\beta-1} \le \tau^\beta$ and  $\tau^{1+\rho} t_{n-1}^{\beta-\rho-1} = \tau^\beta \tau^{1+\rho-\beta} t_{n-1}^{\beta-\rho-1} \le \tau^\beta$, yield
	\begin{equation*}
		\big|\big\langle L_{h,\tau}^n P_h x , P_h x \big\rangle - \Phi_{h,\tau}(x,n)\big| \lesssim \tau^\rho \| x \|^ 2.
	\end{equation*}
	This shows the claim for $n>1$. The case $n=1$ is treated similarly. 
\end{proof}

We now obtain our weak convergence result as a direct consequence of Theorem~\ref{thm:disc_rep} and Corollary~\ref{cor:Phi_bound_Lyap}. We write $\Phi_{h,\tau}(x)= \Phi_{h,\tau}(x,N_\tau)$ for $x \in H$.

\begin{corollary}\label{cor:weak_full}
	Let $\Phi$ and $\Phi_{h,\tau}$ be the functionals given by~\eqref{eq:Phi} and~\eqref{eq:Phi_full}, respectively. For all $c>0$ and $\rho\in(0,\beta)$, there exists a constant $C>0$ satisfying for $h\in(0,1)$, $\tau\leq ch^{2 (1 - \beta)/(1-\rho)}$ and $x \in H$
	\begin{equation*}
		\big|\Phi(x)-\Phi_{h,\tau}(x)\big| \le C T^{-\rho} (h^{2\rho} + \tau^\rho)  \|x\|^2.
	\end{equation*}
\end{corollary}

We conclude this section by relating our approach to prove weak convergence to the most common in the literature. That approach is based a joint use of the It\^o formula and the solution to a Kolmogorov equation, see references in the introduction. For additive noise the solution to the Kolmogorov equation is regular enough to show weak convergence rates. 
For multiplicative noise the solution is less regular and a straight forward generalization of the methodology for additive noise to multiplicative noise leads to suboptimal rates for $\beta\in[1/2,1]$ with finite element approximations, see \cite{AnderssonLarsson2016}. This has been solved in \cite{BrehierDebussche2017,Kurniawan2016} but restricts to spectral methods. In our setting the Kolmogorov equation is solved by the quadratic form of the solution $L$ of the Lyapunov equation with $R=0$, see Theorem~\ref{thm:mild_is_weak}. By Theorem~\ref{thm:Lyapunov} it has the same regularity as for additive noise. Therefore, a weak convergence proof with desired convergence rates could be carried out by adapting the method of \cite{debussche2011}. Our approach is advantageous since it has no regularity assumption on the initial condition as in~\cite{AnderssonKovacsLarsson} and can treat path-dependent functionals, where we are only aware of \cite{AnderssonKovacsLarsson,BHS18} for SPDEs.

\section{Numerical implementation and simulation}
\label{sec:numerics}

The goal of this section is to show how the numerical approximations of Sections~\ref{sec:fully_discrete_Lyapunov} and~\ref{sec:fully_discrete_spde} are implemented in practice. We demonstrate our theoretical results by numerical simulations in the specific setting of Example~\ref{ex:heat_equation}. Solving the Lyapunov equation is then compared to the Monte Carlo method, both by an empirical stability analysis as well as a theoretical computational complexity discussion.

\subsection{Implementation and convergence analysis}
\label{subsec:implementation-simulation-lyapunov}

\subsubsection{Implementation of the fully discrete Lyapunov approximation}

First, we describe how the fully discrete approximation $L^{N_\tau}_{h,\tau}$ from~\eqref{eq:Lyapunov_one_step_A} of the solution $L(T)$ to the Lyapunov equation~\eqref{eq:Lyapunov_var} is implemented numerically. Let $N_h$ denote the dimension of $V_h$ and let $(\phi_h^i)_{i=1}^{N_h}$ be a basis of $V_h$. By $\bM_h, \bA_h, \bG_h$ and $\bR_h \in \R^{N_h \times N_h}$ we denote the matrices with entries given by $(\bM_h)_{i,j} = \langle \phi_h^i, \phi_h^j \rangle$, $(\bA_h)_{i,j} = a(\phi_h^i,\phi_h^j) = \langle A_h^{1/2}\phi_h^i, A_h^{1/2}\phi_h^j \rangle$, $(\bG_h)_{i,j} = \langle G\phi_h^i, G\phi_h^j \rangle$, and $(\bR_h)_{i,j} = \langle R\phi_h^i, R\phi_h^j \rangle$, $i,j \in \{1, 2, \ldots, N_h\}$, respectively. 
For $n \in \{0, \ldots N_\tau\}$, let $\bL^n_{h,\tau}$ be the matrix containing the coefficients in the expansion
$L_{h,\tau}^n=\sum_{i,j=1}^{N_h}(\bL^n_{h,\tau})_{i,j}\phi_h^i\otimes \phi_h^j$ of the approximation given by~\eqref{eq:Lyapunov_one_step_A}. Here $\phi_h^i \otimes \phi_h^j$ is given by $(\phi_h^i \otimes \phi_h^j) \psi_h = \langle \phi_h^j,\psi_h \rangle \phi_h^i$ for all $\psi_h \in V_h$. By $(\bB_h(\bL_{h,\tau}^{n-1}))$ we denote the matrix with entries $(\bB_h(\bL_{h,\tau}^{n-1}))_{i,j}$ given by 
\begin{equation*}
\langle
L_{h,\tau}^{n-1} P_h B(\phi_h^i), B(\phi_h^j)
\rangle_{\LB_2^0}=
\sum_{k,\ell=1}^{N_h}
(\bL^{n-1}_{h,\tau})_{k,\ell}
\langle
B(\phi_h^j) (B(\phi_h^i))^* \phi_h^\ell, \phi_{h}^k
\rangle,
\end{equation*}
where the adjoint of $B(\phi_h^i)$ is taken with respect to $\cL(U,H)$.

The matrix version of~\eqref{eq:Lyapunov_one_step_A} is now given by \begin{equation}
\label{eq:matrix_sylvester_prel}
(\bM_h + \tau \bA_h) \bL_{h,\tau}^n (\bM_h + \tau \bA_h)
=
\bM_h \bL_{h,\tau}^{n-1} \bM_h 
+ \tau \bR_h 
+\tau \bB_h(\bL_{h,\tau}^{n-1})
\end{equation}
with initial value $\bL_{h,\tau}^0 = \bM_h^{-1}\bG_h \bM_h^{-1}$. With this in place, one approximates $\Phi(x)$, $x \in H$, by $
\Phi_{h,\tau}^\mathrm{L}(x) =  \bx_h^\top \bM_h \bL_{h,\tau}^{N_\tau} \bM_h \bx_h$.  
Here, $\bx_h^\top$ denotes the transpose of the vector $\bx_h$ of coefficients $x_h^j$ in the expansion $P_h x = \sum_{j = 1}^{N_h} x_h^j \phi_h^j$.

\subsubsection{Empirical convergence analysis}
\label{subsec:empirical-convergence}

Next, we illustrate, by numerical simulations, our theoretical results (specifically Corollary~\ref{cor:Phi_bound_Lyap} and Corollary~\ref{cor:weak_full}). The same setting as in Example~\ref{ex:heat_equation} is considered, where we recall that $A = - \Delta$. We choose $U=H$, so that the equation is driven by space-time white noise, and assume that $B$ is a linear Nemytskij operator (i.e., that $(B(u)v)(\chi) = u(\chi)v(\chi)$ for all $u,v \in H = L^2(D)$ and almost every $\chi \in D \subset \R^d$). Then $b = \norm[\LB(H,\LB(H,\LB_2^0))]{A^{(\beta-1)/2}B} < \infty$ for all $\beta < 1/2$ in $d=1$ and for no positive $\beta$ when $d>1$. We further specify $D = (0,1)$, $R=0$, $G = \Id$, $T = 1$ and rescale $A$ and $B$ by factors $\lambda = 0.05$ and $\sigma = 0.65$, respectively. We choose an initial value $X_0(\chi) = \chi\mathbf{1}_{[(0,1/2)}(\chi) + (1-\chi)\mathbf{1}_{[1/2,1)}(\chi)$ for $\chi \in D$ and compute $\Phi^{\mathrm{L}}_{h,\tau}(X_0)$ and $\Phi_{h,\tau}(X_0)$ for $\tau = h^2$ and $h = 2^{-1}, 2^{-2}, \ldots, 2^{-8}$. The latter quantity is computed in a deterministic way by tensorizing the matrix equation system corresponding to~\eqref{eq:full_disc_SPDE2} and applying the expectation on both sides. We refer to \cite{P17} for computational details on this tensor product approach. This approach was employed for a reference solution instead of a Monte Carlo method due to the stability problems of the latter, see Section~\ref{sec:empirical-stability-analysis}. The errors $|\Phi^{\mathrm{L}}_{h,\tau}(X_0) - \Phi(X_0)|$ and $|\Phi_{h,\tau}(X_0) - \Phi(X_0)|$ are shown in Figure~\ref{fig:errors}.
\begin{figure}[ht!]
	\centering
	\includegraphics[width = 0.75\textwidth]{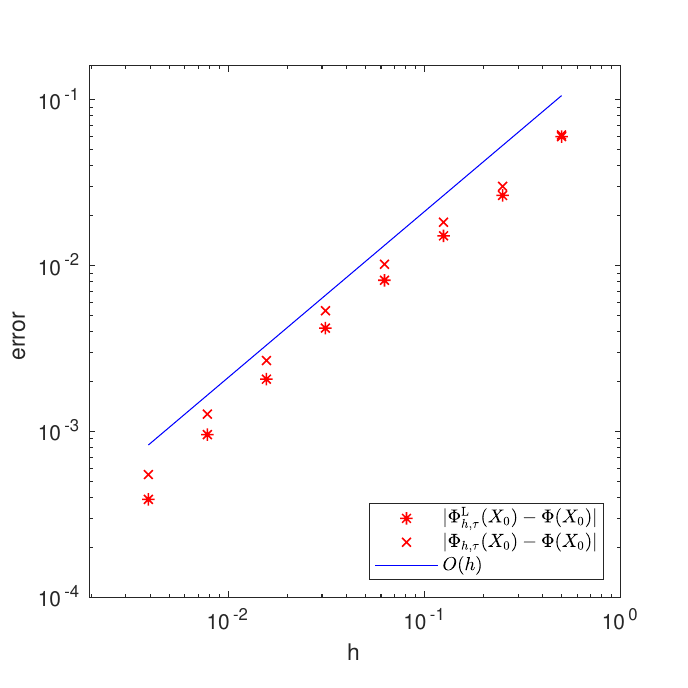}
	\caption{Estimates of the errors $|\Phi^{\mathrm{L}}(X_0) - \Phi^T(X_0)|$ and $|\Phi_{h,\tau}(X_0) - \Phi(X_0)|$ for a fixed $X_0 \in H$.}
	\label{fig:errors}
\end{figure} 
Here we have replaced $\Phi(X_0)$ with the reference solution $\Phi_{h,\tau}(X_0)$ computed with $h = 2^{-10}$ and $\tau = h^2$. All matrices and vectors are computed exactly, which is possible by our choice of $U$, except for the initial value, which is interpolated onto the finite element space. Since $b < \infty$ for all $\beta < 1/2$, we expect from Corollary~\ref{cor:Phi_bound_Lyap} and Corollary~\ref{cor:weak_full} a convergence rate of essentially $\cO(h)$. The results of Figure~\ref{fig:errors} are consistent with this expectation. 

\subsection{Comparison to the Monte Carlo method}
The most straightforward way of utilizing a finite element method for the  approximation of $\Phi(x)$ in~\eqref{eq:Phi}, with $x$ fixed, is by a Monte Carlo method, i.e., by computing
\begin{equation}
\label{eq:mc-def}
	\Phi_{h,\tau,M}^{\mathrm{MC}}
	(x)
	= 
	M^{-1}
	\sum_{j=1}^M
	\left(
	\left\|
	G X_{h,\tau}^{N_\tau, x, (j)}
	\right\|^2
	+
	\sum_{n=0}^{N_{\tau}-1}
	\tau
	\left\|
	R X_{h,\tau}^{n, x, (j)}
	\right\|^2
	\right).
\end{equation}
Here $M$ is the number of iid samples $(X_{h,\tau}^{n, x, (j)})_{j=1}^M$ of $X_{h,\tau}^{n, x}$ computed by the recursion~\eqref{eq:full_disc_SPDE2}. The Monte Carlo method has a low memory requirement, is easy to parallelize and can be improved by multilevel methods (see, e.g., \cite{G15} as well as the comment at the end of the next section). There are, however, certain situations in which the Lyapunov method of computing $\Phi^{\mathrm{L}}_{h,\tau}(x)$ is preferable for the  approximation of $\Phi(x)$, which we now outline.

\subsubsection{Empirical stability analysis}
\label{sec:empirical-stability-analysis}

The zero solution $X(t) = 0$ to an SPDE with multiplicative noise such as~\eqref{eq:SPDE} (or a discretization thereof) can be simultaneously asymptotically stable in a $\P$-a.s.\ sense and unstable in a mean square sense. In a Monte Carlo type method, this results in the number of samples needed for a satisfactory approximation in practice being prohibitively large. We discuss this challenge to such methods in detail, in the setting that $R = 0$ and  $G = \Id$. 

Consider the rescaled semidiscrete stochastic heat equation 
\begin{equation}
\label{eq:rescaled-SPDE}
	\dd X_h(t) + \lambda A_h X_h(t) \, \dd t = \sigma P_h B(X_h (t)) \, \dd W(t).
\end{equation}
Here, $\lambda, \sigma \ge 0$ while all other parameters are as in Section~\ref{subsec:empirical-convergence}. 
Following~\cite{thalhammer2017}, the zero solution to~\eqref{eq:rescaled-SPDE} is said to be \emph{asymptotically stable $\P$-a.s.}\ if, for any $\epsilon \in (0,1)$ and $\epsilon' >0$, there exist
\begin{enumerate}
	\item $\delta > 0$ such that if $\norm{X_h(0)} < \delta$, then $\P(\norm{X_h(t)} > \epsilon') < \epsilon$ for all $t \ge 0$ and
	\item $\delta' > 0$ such that for any initial value $X_h(0)$ satisfying $\norm{X_h(0)} < \delta'$ $\P$-a.s., $\lim_{t\to \infty} \norm{X_h(t)} = 0$ $\P$-a.s.
\end{enumerate}
It is said to be \emph{asymptotically mean square stable} if for every $\epsilon > 0$, there exist
\begin{enumerate}
	\item $\delta > 0$ such that if $\norm{X_h(0)} < \delta$ then $\E[\| X_h(t) \|^2] < \epsilon$ for all $t\geq 0$ and
	\item $\delta' > 0$ such that if $\norm{X_h(0)} < \delta'$ then $\lim_{t \to \infty} \E [ \| X_h(t) \|^2] = 0$.
\end{enumerate}

In~\cite{thalhammer2017}, the authors consider a discretized stochastic heat equation, similar to~\eqref{eq:rescaled-SPDE} but with finite differences instead of finite elements. They prove that as $\sigma$ increases, the zero solution becomes simultaneously asymptotically stable $\P$-a.s.\ and asymptotically mean square unstable. While their results do not directly translate to our setting, we also expect that for large $T$ and $\sigma$, $\E [ \| X_h(T) \|^2]$ becomes very big while most sample paths of $X_h$ are very small, possibly zero within machine accuracy. If the time discretization of $X_h$ shares this property as assumed in~\cite{thalhammer2017}, $\Phi_{h,\tau,M}^{\mathrm{MC}}(X_0)$ approximates $\E [ \| X_h(T) \|^2]$, and therefore also $\E [ \| X(T) \|^2]$, poorly.

\begin{figure}[ht!]
	\centering
	\subfigure[$T=1$ \label{subfig:T1}]{\includegraphics[width = .49\textwidth]{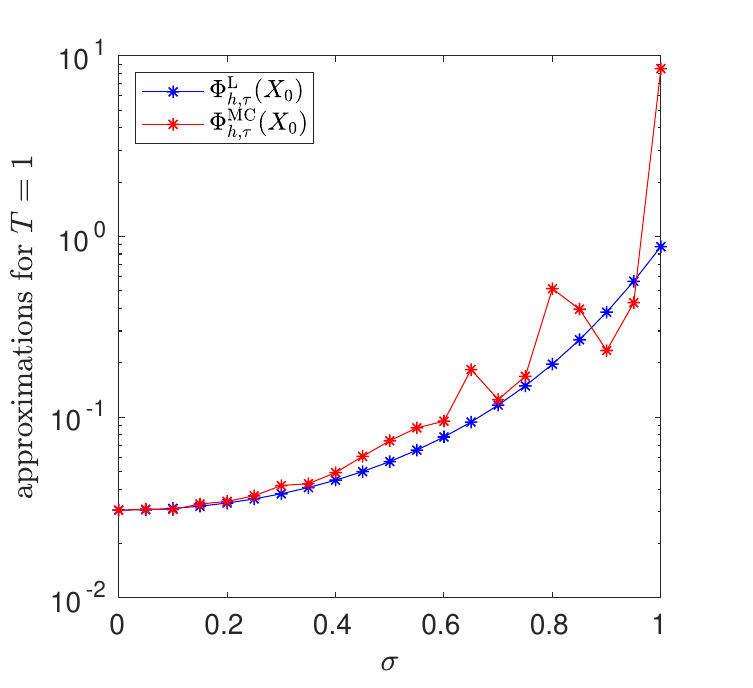}}
	\subfigure[$T=5$ \label{subfig:T5}]{\includegraphics[width = .49\textwidth]{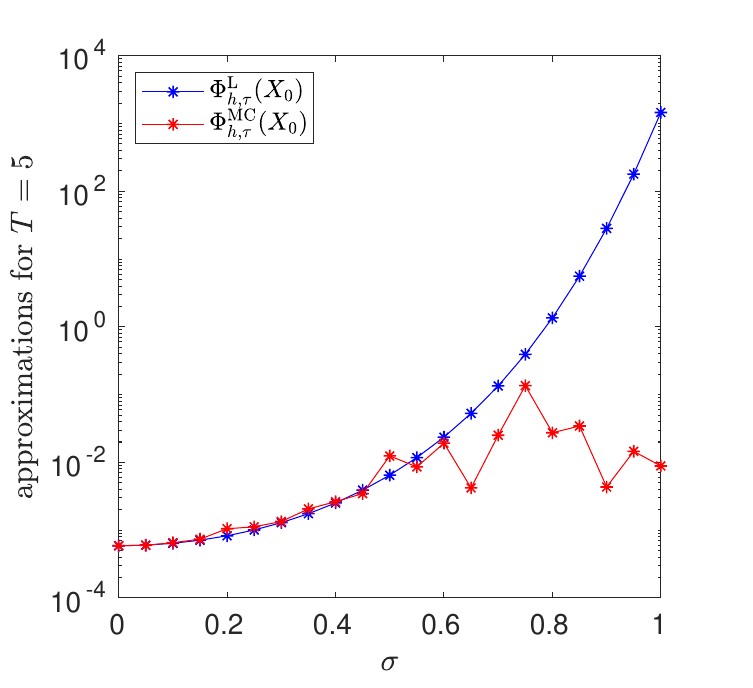}}
	\subfigure[$T=10$ \label{subfig:T10}]{\includegraphics[width = .49\textwidth]{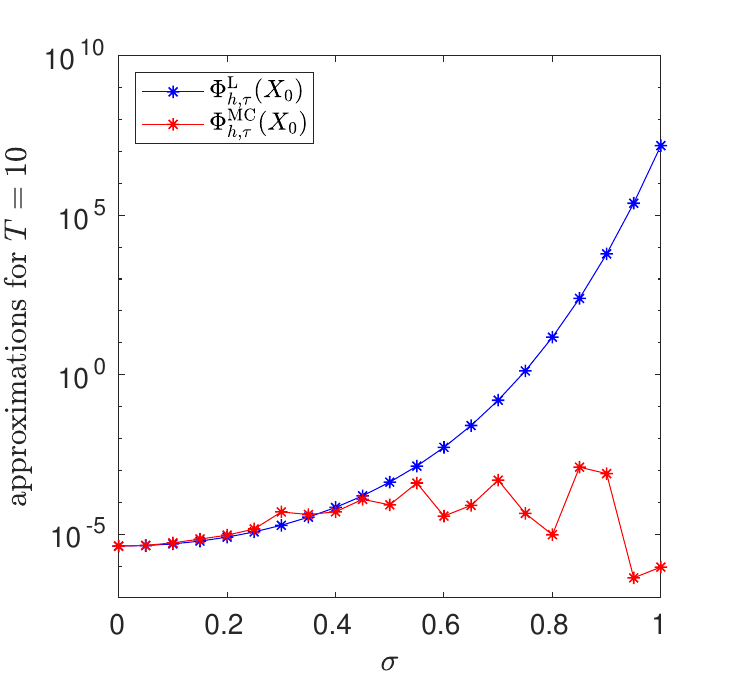}}
	\subfigure[$T=10$, 10-fold increase of sample size \label{subfig:T10-10fold}]{\includegraphics[width = .49\textwidth]{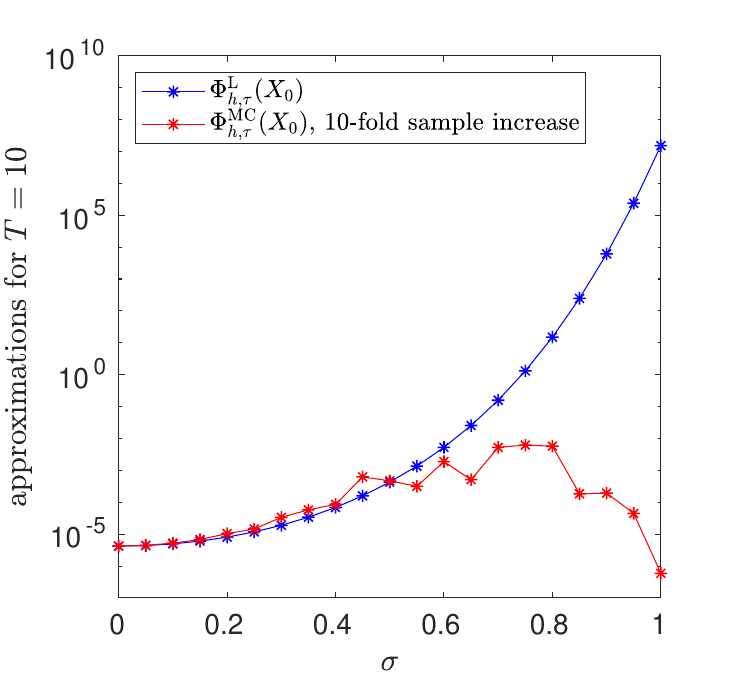}}
	\caption{Lyapunov and Monte Carlo approximations of $\mathbb{E}[\norm{X(T)}^2]$ for $T = 1, 5$ and $10$ and $\sigma \in [0,1]$.}
\end{figure} 

To investigate this in practice, we choose $X_0$ and $\lambda=0.05$ as in Section~\ref{subsec:empirical-convergence} and compute approximations of $\E [ \| X(T) \|^2]$ for various values of $\sigma$ and $T$. The results from the Lyapunov method and the Monte Carlo method are compared. We choose $h = 2^{-5}$ and $\tau = h^2$ for both methods, with $M = h^{-2} = 1024$ samples in the Monte Carlo method. In Figure~\ref{subfig:T1} we see the results of the two methods for $T=1$ and $\sigma \in [0,1]$. They seem to agree reasonably well for small values of $\sigma$, while the variance of $\Phi_{h,\tau,M}^{\mathrm{MC}}(X_0)$ increases as $\sigma$ increases. For moderately larger values of $T$, we already see the consequences of the mean square instability and the $\P$-a.s.\ stability of the zero solution to~\eqref{eq:rescaled-SPDE}. In Figure~\ref{subfig:T5}, with $T=5$, the difference between the two approximations is of several orders of magnitude. This behavior is even more pronounced for $T=10$ in Figure~\ref{subfig:T10}. Even if we increase the number of samples to $M = 10 h^{-2} $ (see Figure~\ref{subfig:T10-10fold}), the results do not improve for $\sigma \approx 1$. In other words, even for moderately large values of $\sigma$ and $T$, the Monte Carlo method fails to give reliable results. As noted in~\cite{thalhammer2017}, replacing this simple Monte Carlo estimator with a multilevel Monte Carlo method does not solve this issue. Indeed, if empirical variances are used to estimate the number of samples needed at each level this method can suffer even more from this stability problem.

\subsubsection{Computational complexity comparison}

Even under parameter choices for which the stability problems outlined in the previous section do not occur, there may be other reasons for why one would prefer to approximate $\E [ \| X(T) \|^2]$ by means of the Lyapunov method rather than by Monte Carlo. First, the computation of $\Phi_{h,\tau,M}^{\mathrm{MC}}
(x)$ can be expensive if $R$ is a non-local operator. Then the matrix $\bR_h $ is typically dense. This matrix is applied for the computation of the term containing $R$ in~\eqref{eq:mc-def} a total of $M \times N_{\tau}$ times. By the law of large numbers, $M$ should be chosen to be proportional to the inverse of the square root of the weak error in Corollary~\ref{cor:weak_full} in order for the Monte Carlo error not to asymptotically dominate the full mean squared error. Therefore, the computational cost can become prohibitively large.

Second, the Lyapunov method yields a way of approximating $\Phi(x)$ for all $x \in H$ simultaneously. The Monte Carlo method can be adapted to this setting by iterating~\eqref{eq:full_disc_SPDE2} to obtain 
\begin{equation*}
X_{h,\tau}^{n,x}
=
\left(\prod_{j=0}^{n-1} S_{h,\tau} \left(P_h + B P_h (\cdot) \Delta W^j \right)\right)
P_h x =: F^{n}_{h,\tau} P_h x
\end{equation*}
and then computing 
\begin{align*}
&\Phi_{h,\tau,M}^{\mathrm{MC}}
(\cdot) \\
&\qquad=
\left\langle
M^{-1}
\sum_{j=1}^M
\left(
(F^{N_\tau,(j)}_{h,\tau})^* G^* G F^{N_\tau,(j)}_{h,\tau}+
\sum_{n=0}^{N_{\tau}-1}
\tau
(F^{n,(j)}_{h,\tau})^* R^* R F^{n,(j)}_{h,\tau}
\right)
P_h \cdot,
\cdot\right\rangle.
\end{align*}
However, forming the matrix corresponding to the sum of operators requires, again, the multiplication and addition of dense matrices, leading to great costs in terms of both computational power and memory. In fact, in the setting of the simulations of Section~\ref{subsec:empirical-convergence}, it can be seen that the computational cost of $\Phi^{\mathrm{L}}_{h,\tau}$ is $\cO(h^{-4})$. This can be compared to $\cO(h^{-6})$ for $\Phi_{h,\tau,M}^{\mathrm{MC}}(\cdot)$, the cost of the iterated Monte Carlo method if one chooses $M \simeq h^{-2}$ so that the additional Monte Carlo error does not dominate. 

In conclusion, we have demonstrated that there may be several situations in which the Lyapunov method of computing $\Phi^{\mathrm{L}}_{h,\tau}(X_0)$ is preferable for the  approximation of $\Phi(X_0)$ compared to a Monte Carlo method. It should, however, be noted that it may be fruitful to combine these methods. For example, the Lyapunov method, computed at a coarse grid, may be used to indicate the presence of mean square instability combined with $\P$-a.s.\ stability in a system such as~\eqref{eq:rescaled-SPDE}, suggesting that remedies such as the ones discussed in~\cite{thalhammer2017} should be taken.

\bibliography{lit}
\bibliographystyle{myplain}

\appendix
\section{Proofs for Section~\ref{subsec:exLyapunovSol}}

\subsection{Proof of Theorem~\ref{thm:Lyapunov}}
\label{sec:eu}

As a first step we prove a result that is needed to ensure that the integrals in the proof of Theorem~\ref{thm:Lyapunov} are well-defined. The setting and notation is the same as in Sections~\ref{sec:setting} and~\ref{sec:Lyapunov_equation_and_SPDE} of the main text.

 \begin{lemma}\label{lemma:strong_cont}
For all $t\in \IT_0$, $\theta_1, \theta_2\geq0$ and $\Upsilon\in \cV$ the mappings
\begin{equation*}
	[0,t)\ni s \mapsto A^{\frac{\theta_1}2}S(t-s)R^*RS(t-s) A^{\frac{\theta_2}2} \in \LB(H)
\end{equation*}
and
\begin{equation*}
(0,t)\ni s \mapsto A^{\frac{\theta_1}2}S(t-s)B^* \Upsilon(s)BS(t-s) A^{\frac{\theta_2}2} \in \LB(H)
\end{equation*}
are strongly continuous.
\end{lemma}

\begin{proof}
We prove the strong continuity of the second mapping in detail. The proof of the first is done in the same way. 

The second mapping is a composition of the mappings $(0,t)\ni s\mapsto A^{\theta_1/2}S(t-s)\in \LB(H)$, $(0,t)\ni s\mapsto B^*\Upsilon(s)B\in \LB(H)$ and $(0,t)\ni s\mapsto S(t-s)A^{\theta_2/2}\in \LB(H)$. The first and third are strongly continuous as mappings $(0,t)\to\LB(\dot H^{\theta_1} ,H)$ and $(0,t)\to \LB(\dot H^{\theta_2},H)$, respectively. This property can be extended to $(0,t)\to\LB(H)$ in both cases by considering an approximating sequence in $\dot H^{\theta_1}$ and $\dot H^{\theta_2}$, respectively. Using \eqref{eq:S_smooth}, the strong continuity of $S$ and uniform continuity of $\Upsilon$ we obtain for $s,\tilde{s} \in (0,t)$ and $\phi \in H$ that
\begin{align*}
&\big\|
    \big(
       A^{\frac{\theta_1}2}S(t-s)B^* \Upsilon(s)BS(t-s) A^{\frac{\theta_2}2}
      -
      A^{\frac{\theta_1}2}S(t-\tilde{s})B^* \Upsilon(\tilde{s})BS(t-\tilde{s}) A^{\frac{\theta_2}2}
    \big)\phi
  \big\|\\
&\qquad
  \leq
  \big\|
    \big(
      A^{\frac{\theta_1}2}S(t-s)-A^{\theta_1}S(t-\tilde{s})\big)B^* \Upsilon(s)BS(t-s) A^{\frac{\theta_2}2}    \phi
  \big\|\\  
&\qquad\quad
  +
  \big\|
      A^{\frac{\theta_1}2}S(t-\tilde{s})B^*( \Upsilon(s)-\Upsilon(\tilde{s}))BS(t-s) A^{\frac{\theta_2}2}\phi
  \big\|\\  
&\qquad\quad
  +
  \big\|
      A^{\frac{\theta_1}2}S(t-\tilde{s})B^*\Upsilon(\tilde{s})B
      \big(
        S(t-s) A^{\frac{\theta_2}2}
        -
        S(t-\tilde{s}) A^{\frac{\theta_2}2}
      \big)\phi
  \big\|\\
&\qquad
  \leq
  \big\|
    \big(
      A^{\frac{\theta_1}2}S(t-s)-A^{\frac{\theta_1}2}S(t-\tilde{s})
    \big)
    B^* \Upsilon(s)BS(t-s) A^{\frac{\theta_2}2}\phi
  \big\|\\  
&\qquad\quad
  +
  b^2
  C_{\theta_1}C_{\theta_2}
  (t-\tilde{s})^{-\frac{\theta_1}2} (t-s)^{-\frac{\theta_2}2}
  \big\|A^{\frac{1-\beta}2}( \Upsilon(s)-\Upsilon(\tilde{s}))A^{\frac{1-\beta}2}\big\|_{\LB(H)}
  \| \phi \|\\  
&\qquad\quad
  +
  b^2 C_{\theta_1}(t-\tilde{s})^{-\frac{\theta_1}2}\tilde{s}^{\beta-1}
  \triple
    \Upsilon
  \triple_0
  \big\|
      \big(
        S(t-s) A^{\frac{\theta_2}2}
        -
        S(t-\tilde{s}) A^{\frac{\theta_2}2}
      \big)\phi
  \big\|
\end{align*}
with the constant notation introduced in Section~\ref{sec:setting}.
Since $B^* \Upsilon(s)BS(t-s) A^{\theta_2/2}\phi \in H$ and $S$ is strongly continuous and since $\Upsilon\colon\IT_0\to \LB(\dot{H}^{\beta-1},\dot{H}^{1-\beta})$ is uniformly continuous, the right hand side converges to zero as $\tilde{s}$ tends to $s$. This completes the proof.
\end{proof}

The rest of the proof of Theorem~\ref{thm:Lyapunov} is based on a global fixed point argument in $\cV$. 
Equation \eqref{eq:Lyapunov_mild} is written in the form of a fixed point equation 
\begin{equation*}
   L
=
 \cH(L)
 = \cI
  +
  \cJ
  +
  \cK(L),
\end{equation*}
with $\cI, \cJ \in \cV$, $\cK \colon \cV \to \cV$ acting on $\phi\in H$ by
\begin{align*}
\cI(t) \phi
  & = S(t)G^*GS(t) \phi,\\
\cJ(t) \phi
  & =
    \int_0^t
      S(t-s)R^*RS(t-s) \phi
    \diff s,\\
\cK(\Upsilon)(t) \phi
  & =
    \int_0^t
      S(t-s)B^* \Upsilon(s) BS(t-s) \phi
    \diff s.
\end{align*}
Existence of a unique solution in~$\cV$ follows by the Banach fixed point theorem by  proving that $\cH$ is well-defined and that  for some $\sigma>0$ the fixed point map $\cH$ is a contraction, i.e., that there exists $\eta \in (0,1)$ such that for all $\Upsilon_1,\Upsilon_2\in \cV$
\begin{equation}\label{eq:contraction}
 \triple \cH(\Upsilon_1) - \cH(\Upsilon_2) \triple_{\sigma}
    = \triple \cK(\Upsilon_1) - \cK(\Upsilon_2) \triple_{\sigma}
    = \triple \cK(\Upsilon_1 - \Upsilon_2) \triple_{\sigma}
    \le \eta \, \triple \Upsilon_1 - \Upsilon_2 \triple_{\sigma}.
\end{equation}
The proof is organized as follows: We start by proving that $\triple\cI\triple_0+\triple\mathcal{J}\triple_0+\triple\cK\triple_0<\infty$ and continue by showing that $\cI$, $\cJ$ and $\cK(\Upsilon)$ with  $\Upsilon\in \cV$ are strongly and uniformly continuous on $\IT$ and $\IT_0$, respectively. From this we conclude that $\cI$, $\cJ$, $\cK$ are well-defined and derive bounds to show the contraction property~\eqref{eq:contraction} of~$\cH$ and the claimed regularity estimates.

To prove that $\triple\cI\triple_0<\infty$ we observe, using \eqref{eq:S_smooth}, that for all $\theta_1,\theta_2\geq0$ and $t\in\IT_0$
\begin{equation}\label{eq:I}
  \big\|
    A^{\frac{\theta_1}2}
    \cI(t)
    A^{\frac{\theta_2}2}
  \big\|_{\LB(H)}
\leq
  g^2
  \big\|
    A^{\frac{\theta_1}2}S(t)
  \big\|_{\LB(H)}
  \big\|
    S(t)A^{\frac{\theta_2}2}
  \big\|_{\LB(H)}
  \leq
  g^2C_{\theta_1}C_{\theta_2} t^{-\frac{\theta_1+\theta_2}2}.
\end{equation}
Setting $\theta_1=\theta_2=1-\beta$ and $\theta_1=\theta_2=0$, respectively, shows the desired bound 
\begin{equation}\label{eq:I_finite}
\triple
    \cI
  \triple_0
  \leq g^2(C_0^2 + C_{1-\beta}^2)<\infty.
\end{equation}
For $\cJ$ we use~\eqref{eq:S_smooth} to obtain that for all $\theta_1,\theta_2\geq0$, $t\in\IT_0$, $s\in[0,t)$ and $\phi\in H$
\begin{equation}\label{eq:J0.1}
  \big\|
    A^{\frac{\theta_1}2}S(t-s)R^*RS(t-s)A^{\frac{\theta_2}2} \phi
  \big\|
  \leq
  r^2 C_{\theta_1} C_{\theta_2} \|\phi\| (t-s)^{-\frac{\theta_1+\theta_2}2}.
\end{equation}
Together with the strong continuity of the mapping shown in Lemma~\ref{lemma:strong_cont}, this bound implies that 
$A^{\theta_1/2}S(t-\cdot)R^*RS(t-\cdot)A^{\theta_2/2} \phi \in L^1([0,t],H)$ 
when $\theta_1,\theta_2\in[0,2)$ with $\theta_1+\theta_2<2$, 
and in particular that the Bochner integral in $\cJ$ is well-defined. It also shows that $\cJ$ can be extended to a mapping $\IT_0\to \LB(\dot{H}^{-\theta_2},\dot{H}^{\theta_1})$. Next, for $t\in\IT_0$
\begin{align}\label{eq:J0.2}
\begin{split}
	\big\|
	A^{\frac{\theta_1}2}
	\cJ(t)
	A^{\frac{\theta_2}2}
	\big\|_{\LB(H)} 
	&= \sup_{\phi \in H, \| \phi\| = 1} 
	\Big\| 
	\int_0^t
	A^{\frac{\theta_1}2}S(t-s)R^*RS(t-s)A^{\frac{\theta_2}2} \phi
	\diff s
	\Big\| \\
	&\le 
	\int_0^t
	 \sup_{\phi \in H, \| \phi\| = 1}  
	\Big\| 
	A^{\frac{\theta_1}2}S(t-s)R^*RS(t-s)A^{\frac{\theta_2}2} \phi
	\Big\| 
	\diff s \\
	&= \int_0^t
	\Big\| 
	A^{\frac{\theta_1}2}S(t-s)R^*RS(t-s)A^{\frac{\theta_2}2}
	\Big\|_{\LB(H)} 
	\diff s.
\end{split}
\end{align}
 Using \eqref{eq:J0.1} we have for $t\in\IT_0$
\begin{equation}\label{eq:J}
\big\|
    A^{\frac{\theta_1}2}
    \cJ(t)
    A^{\frac{\theta_2}2}
  \big\|_{\LB(H)}
\leq
  \frac{
    2r^2C_{\theta_1}C_{\theta_2} 
  }{
    2-\theta_1-\theta_2
  } \,
  t^{1-\frac{\theta_1-\theta_2}2}.
\end{equation}
Setting $\theta_1=\theta_2=1-\beta$ and $\theta_1=\theta_2=0$, respectively, shows that 
\begin{equation}\label{eq:J_finite}
  \triple
    \cJ
  \triple_0
  \leq
  r^2
  \left(
    C_0^2 T + C_{1-\beta}^2 \beta^{-1} T^\beta
  \right) 
  <\infty.
\end{equation}
We now turn our attention to $\cK$. For all $\theta_1,\theta_2\geq0 $, $t\in\IT_0$, $s\in(0,t)$, $\Upsilon\in\cV$, and $\phi\in H$, we have by using \eqref{eq:S_smooth} that
\begin{align}\label{eq:K0.1}
\begin{split}
&\big\|
    A^{\frac{\theta_1}2}S(t-s)B^* \Upsilon(s)BS(t-s) A^{\frac{\theta_2}2}\phi
  \big\|\\
&\qquad
  \leq
  b^2
  C_{\theta_1}
  C_{\theta_2}
  \triple
    \Upsilon
  \triple_{\sigma}
  \|\phi\|
  e^{\sigma s}
  s^{\beta-1} (t-s)^{-\frac{\theta_1+\theta_2}2}.
\end{split}
\end{align}
Combining this with Lemma~\ref{lemma:strong_cont} implies that for $\theta_1,\theta_2\in[0,2)$ with $\theta_1+\theta_2<2$, $A^{\theta_1/2}S(t-\cdot)B^* \Upsilon(\cdot)BS(t-\cdot) A^{\theta_2/2}\phi \in L^1([0,t],H)$
and in particular that the Bochner integral in $\cK$ is well-defined. Similarly to $\cJ$, it also implies that $\cK(\Upsilon)$ can be extended to $\IT\to \LB(\dot{H}^{-\theta_2},\dot{H}^{\theta_1})$. Using \eqref{eq:K0.1} we obtain, similarly to \eqref{eq:J0.2}, the bound
\begin{equation}\label{eq:K0.2}
\big\|
    A^{\frac{\theta_1}2}
    \cK(\Upsilon)(t)
    A^{\frac{\theta_2}2}
  \big\|_{\LB(H)}
  \leq
  b^2C_{\theta_1}C_{\theta_2}
  \Bigg(
    \int_0^t
      e^{\sigma s}
      s^{\beta-1}
      (t-s)^{-\frac{\theta_1+\theta_2}2}
    \diff s
  \Bigg)
  \triple
    \Upsilon
  \triple_\sigma.
\end{equation}
For the case $\sigma = 0$, we conclude from~\eqref{eq:beta_integral} that
\begin{equation}\label{eq:K}
  \big\|
    A^{\frac{\theta_1}2}
    \cK(\Upsilon)(t)
    A^{\frac{\theta_2}2}
  \big\|_{\LB(H)}
\leq
  b^2C_{\theta_1}C_{\theta_2}
  t^{\beta-\frac{\theta_1+\theta_2}2}
  \B \Big(\beta,1-\frac{\theta_1+\theta_2}2\Big)
  \triple
    \Upsilon
  \triple_0.
\end{equation}
Setting again $\theta_1=\theta_2=1-\beta$ and $\theta_1=\theta_2=0$, respectively, yields
\begin{equation}\label{eq:K_finite}
  \triple
    \cK(\Upsilon)
  \triple_0
  \leq
  b^2 
  \Big(
    C_0^2 \B(\beta,1) T
    +
    C_{1-\beta}^2 \B(\beta,\beta) T^\beta
  \Big)
  \triple
    \Upsilon
  \triple_0
  <\infty
\end{equation}
and combining \eqref{eq:I_finite}, \eqref{eq:J_finite}, \eqref{eq:K_finite} implies that $\triple \cH(\Upsilon) \triple_0<\infty$.

Next we show continuity of $\cI$, $\cJ$ and $\cK(\Upsilon)$ on $\IT_0$ and $\IT$. For this purpose we prove H\"older continuity in operator norms on $(0,T)$ and strong continuity at zero separately. The Hölder continuity also implies \ref{eq:temp_reg} once existence and uniqueness have been established. For all $\theta_1,\theta_2\geq0$, $\xi \in[0,1]$ and $t_1,t_2\in\IT_0$ with $t_1<t_2$, we bound using \eqref{eq:S_smooth}, \eqref{eq:S_Holder1} and the semigroup property of $S$ 
\begin{align}\label{eq:Holder_I}
\begin{split}
&\big\|
    A^{\frac{\theta_1}2}
    \big(
      \cI(t_2)-\cI(t_1)
    \big)
    A^{\frac{\theta_2}2}
  \big\|_{\LB(H)}\\
&\quad
  \leq
  g^2
  \big\|
    A^{\frac{\theta_1}2}
    S(t_1)
    A^\xi
  \big\|_{\LB(H)}
  \big\|
    A^{-\xi}
    \big(
      S(t_2-t_1)-\Id
    \big)
  \big\|_{\LB(H)}
  \big\|
    S(t_2)
    A^{\frac{\theta_2}2}
  \big\|_{\LB(H)}\\
&\qquad
  +
  g^2
  \big\|
    A^{\frac{\theta_1}2}
    S(t_1)
  \big\|_{\LB(H)}
  \big\|
    \big(
      S(t_2-t_1)
      -
      \Id
    \big)
    A^{-\xi}
  \big\|_{\LB(H)}
  \big\|
    A^{\xi}
    S(t_1)
    A^{\frac{\theta_2}2}
  \big\|_{\LB(H)}\\
&\quad
  \leq
  g^2 
  C_{2\xi}
  \big(
    C_{\theta_1 + 2\xi} C_{\theta_2}
    +
    C_{\theta_1} C_{\theta_2 + 2\xi} 
  \big)
  t_1^{-\frac{\theta_1+\theta_2}2-\xi}
  |t_2-t_1|^{\xi}.
\end{split}
\end{align}
This shows in particular strong continuity $\cI\in\cC_{\mathrm{s}}(\IT_0, \LB(H))$ and uniform continuity $\cI \in \cC(\IT_0,\LB(\dot{H}^{\beta-1}, \dot{H}^{1-\beta}))$. To prove $\cI\in \cV$ it remains to prove strong continuity $\cI \in \cC_{\mathrm{s}}(\IT, \LB(H))$ and for this we have by the strong continuity~\eqref{eq:S_Holder1} of $S$ for $\phi\in H$
\begin{align*}
  \lim_{t\to 0}
    \big\| 
      (\cI(t) - \cI(0) ) \phi
    \big\|
&\le 
  \lim_{t\to 0}
  \big\|
    (S(t) - \Id)G^*GS(t)\phi
  \big\|\\
&\quad
  +
  \lim_{t\to 0}
  \big\|
    G^*G(S(t)-\Id) \phi
  \big\|
  =
  0.
\end{align*}

Let us continue with $\cK$. We observe that for $\phi\in H$, $\Upsilon\in\cV$ and $t_1,t_2\in\IT_0$ with $t_1<t_2$ that
\begin{align*}
&\big(
    \cK(\Upsilon)(t_2) - \cK(\Upsilon)(t_1)
  \big)\phi\\
&\qquad
  =
  \int_0^{t_1}
    (
      S(t_2-s) - S(t_1-s)
    )
    B^* \Upsilon(s) B
    S(t_2-s)\phi
  \diff s\\
& \qquad\quad
  +
  \int_0^{t_1}
    S(t_1-s)
    B^* \Upsilon(s) B
    (
      S(t_2-s) - S(t_1-s)
    )\phi
  \diff s\\
&\qquad\quad
  +
  \int_{t_1}^{t_2}
    S(t_2-s)B^* \Upsilon(s) BS(t_2-s)\phi
  \diff s.
\end{align*}
Similarly to \eqref{eq:J0.2}, this yields by using the semigroup property of $S$ and \eqref{eq:S_Holder1} that
\begin{align*}
&\big\|
    A^{\frac{\theta_1}2}
    \big(
      \cK(\Upsilon)(t_2)-\cK(\Upsilon)(t_1)
    \big)
    A^{\frac{\theta_2}2}
  \big\|_{\LB(H)}\\
&\qquad\leq
  b^2
  \int_0^{t_1}
    \big\|
      A^{\frac{\theta_1}2}
      S(t_1-s)
      A^{\xi}
    \big\|_{\LB(H)}
    \big\|
      A^{-\xi}
      \big(
        S(t_2-t_1)-\Id
      \big)
    \big\|_{\LB(H)}\\
&\qquad\qquad\qquad\quad\times
    \big\|
      A^{\frac{1-\beta}2}
      \Upsilon(s)
      A^{\frac{1-\beta}2}
    \big\|_{\LB(H)}
    \big\|
        S(t_2-s)
      A^{\frac{\theta_2}2}
    \big\|_{\LB(H)}
  \diff s\\
&\qquad\quad
  +
  b^2
  \int_0^{t_1}
    \big\|
        A^{\frac{\theta_1}2}
        S(t_1-s)     
    \big\|_{\LB(H)}
    \big\|
      A^{\frac{1-\beta}2}
      \Upsilon(s)
      A^{\frac{1-\beta}2}
    \big\|_{\LB(H)}\\
&\qquad\qquad\qquad\quad\times    
    \big\|
      S(t_1-s)
      A^{\xi+\frac{\theta_2}2}
    \big\|_{\LB(H)}
    \big\|
      A^{-\xi}
      \big(
        S(t_2-t_1)-\Id
      \big)
    \big\|_{\LB(H)}
  \diff s\\
&\qquad\quad
  +
  b^2
  \int_{t_1}^{t_2}
    \big\|
      A^{\frac{\theta_1}2}
      S(t_2-s)
    \big\|_{\LB(H)}
    \big\|
      A^{\frac{1-\beta}2}
      \Upsilon(s)
      A^{\frac{1-\beta}2}
    \big\|_{\LB(H)}\\
& \qquad\qquad\qquad\quad\times    
    \big\|
      S(t_2-s)
      A^{\frac{\theta_2}2}
    \big\|_{\LB(H)}
  \diff s
\end{align*}
and thus using \eqref{eq:S_smooth} and \eqref{eq:S_Holder1}
\begin{align*}
&\big\|
    A^{\frac{\theta_1}2}
    \big(
      \cK(\Upsilon)(t_2)-\cK(\Upsilon)(t_1)
    \big)
    A^{\frac{\theta_2}2}
  \big\|_{\LB(H)}\\
&\qquad\leq
  b^2
  C_\xi
  C_{\theta_1+2\xi}
  C_{\theta_2}
  \triple
    \Upsilon
  \triple_0
  \Bigg(
    \int_0^{t_1}
      s^{\beta-1}(t_1-s)^{-\frac{\theta_1}2-\xi}(t_2-s)^{-\frac{\theta_2}2}
    \diff s
  \Bigg)
  |t_2-t_1|^\xi \\ 
  &\qquad\quad
  +
  b^2
  C_{2\xi}
  C_{\theta_1}
  C_{\theta_2+2\xi}
  \triple
    \Upsilon
  \triple_0
  \Bigg(
    \int_0^{t_1}
      s^{\beta-1}(t_1-s)^{-\frac{\theta_1+\theta_2}2-\xi}
    \diff s
  \Bigg)
  |t_2-t_1|^\xi \\ 
  &\qquad\quad
  +
  b^2 C_{\theta_1} C_{\theta_2}
  \triple
  \Upsilon
  \triple_0
  \int_{t_1}^{t_2}
  s^{\beta-1}(t_2-s)^{-\frac{\theta_1+\theta_2}2}
  \diff s.
\end{align*}
Since for $s\in[0,t_1)$, $(t_2-s)^{-\theta_2/2} \leq (t_1-s)^{-\theta_2/2}$ and for $s\in(t_1,t_2]$,
$
  s^{\beta-1}
  \leq
  T^\beta
  t_1^{-\xi-(\theta_1+\theta_2)/2} s^{\xi + (\theta_1 + \theta_2)/2 - 1}
  \leq
  T^\beta
  t_1^{-\xi-(\theta_1+\theta_2)/2} (s-t_1)^{\xi + (\theta_1 + \theta_2)/2 - 1}
$,
we have
\begin{align*}
&\big\|
    A^{\frac{\theta_1}2}
    \big(
      \cK(\Upsilon)(t_2)-\cK(\Upsilon)(t_1)
    \big)
    A^{\frac{\theta_2}2}
  \big\|_{\LB(H)}\\
&\,\leq
  b^2
  C_{2\xi}
  \big(
    C_{\theta_1 + 2\xi}
    C_{\theta_2}
    +
    C_{\theta_1}
    C_{\theta_2+2\xi}
  \big)    
  \triple
    \Upsilon
  \triple_0
  \Bigg(
    \int_0^{t_1}
      s^{\beta-1}(t_1-s)^{-\frac{\theta_1+\theta_2}2-\xi}
    \diff s
  \Bigg)
  |t_2-t_1|^\xi \\ 
  &\,\quad
  +
  b^2 C_{\theta_1} C_{\theta_2}
  \triple
  \Upsilon
  \triple_0
  T^\beta
  t_1^{-\xi-\frac{\theta_1+\theta_2}2} 
  \int_{t_1}^{t_2}
    (s-t_1)^{\xi + \frac{\theta_1 + \theta_2}2 - 1}
    (t_2-s)^{-\frac{\theta_1+\theta_2}2}
  \diff s.
\end{align*}
Using \eqref{eq:beta_integral} we therefore obtain
\begin{equation}\label{eq:Holder_K}
\big\|
A^{\theta_1}
\big(
\cK(\Upsilon)(t_2)-\cK(\Upsilon)(t_1)
\big)
A^{\theta_2}
\big\|_{\LB(H)} \lesssim
  t_1^{-\frac{\theta_1+\theta_2}2-\xi}|t_2-t_1|^\xi.
\end{equation}
This implies the desired continuity on~$\IT_0$. For the continuity at zero, we use \eqref{eq:S_smooth} to see that
\begin{equation*}
  \big\| 
    (\cK(\Upsilon)(t) - \cK(\Upsilon)(0) ) \phi
  \big\|
=
  \Big\|
    \int_0^t
      S(t-s)B^* \Upsilon(s) BS(t-s) \phi
    \diff s
    \Big\|
  \leq
  \frac{
    b^2
    C_0^2t^\beta   
    \triple
      \Upsilon
    \triple_0
  }{\beta}
\end{equation*}
and as a consequence 
$
  \lim_{t\to 0}\| 
    (\cK(\Upsilon)(t) - \cK(\Upsilon)(0) ) \phi
  \| = 0
$.
We conclude that $\cK(\Upsilon)\in \cV$. The proof of $\cJ\in \cV$ is similar and therefore omitted. In conclusion we have shown that the fixed point map $\cH$ is well-defined.

It remains to prove the contraction property~\eqref{eq:contraction}. For $\sigma\in\R$ the same arguments as in the proof of \cite[Theorem~2.9]{AJK21} imply that for all $\lambda\in[0,1)$ 
\begin{equation*}
\lim_{\sigma\uparrow\infty}
  \Bigg(
    \sup_{t\in\IT_0}
    t^{1-\beta}
    \int_0^t
      e^{-\sigma (t-s)}
      s^{\beta-1}
      (t-s)^{-\lambda}
    \diff s
  \Bigg)
  =
  0.
\end{equation*}
Combining this with~\eqref{eq:K0.2} implies the existence of $\sigma > 0$ and $\eta\in(0,1)$ such that
\begin{align*}
  \triple
    \cK(\Upsilon)
  \triple_\sigma
&\leq
  b^2
  \Bigg(
    C_0^2
    \sup_{t\in\IT_0}
    \int_0^t
      e^{-\sigma (t-s)}
      s^{\beta-1}
    \diff s\\
&\qquad
    +
    C_{1-\beta}^2
    \sup_{t\in\IT_0}
    t^{1-\beta}
    \int_0^t
      e^{-\sigma (t-s)}
      s^{\beta-1}
      (t-s)^{\beta-1}
    \diff s
  \Bigg)
  \triple
    \Upsilon
  \triple_\sigma
 \le \eta \triple \Upsilon \triple_\sigma.
\end{align*}
We have therefore shown that $\cH$ is a contraction with respect to the $\triple\cdot\triple_\sigma$-norm for sufficiently large $\sigma>0$. The Banach fixed point theorem guarantees the existence and uniqueness of a fixed point $L$ to the mapping $\cH$. This is the unique mild solution to \eqref{eq:Lyapunov_mild}. To prove that $L(\IT)\subset \Sigma(H)$ we consider $\cV' = \Cs(\IT,\Sigma(H)) \cap \cC(\IT_0,\LB(\dot H^{\beta-1},\dot H^{1-\beta})) \subset \cV$, which is a Banach subspace of $\cV$ since $\Sigma(H)$ is a closed subspace of $\cL(H)$, along with the restriction $\mathcal{H}'\colon \cV' \to \cV'$ of the fixed point map $\cH\colon \cV \to \cV$ to $\cV'$. It is a contraction with the same $\triple \cdot\triple_{\sigma}$-norm as $\cH$. One easily checks that $\mathcal{H}'(\Upsilon')(t)$ is self-adjoint for self-adjoint $\Upsilon' \in \cV'$ and thus $\mathcal{H}'$ is well-defined.  Therefore, a second application of the Banach fixed point theorem yields a unique $L'\in\cV'$ such that $\mathcal{H}'(L') = L'$. Since $\cV'\subset\cV$ and $L\in\cV$ is unique, $L=L'$ and thus $L(\IT)\subset \Sigma(H)$. Moreover, Theorem~\ref{thm:Phi} in the end of this section shows that $L(\IT)\subset \Sigma^+(H)$ and we stress that none of the proofs of Theorem~\ref{thm:mild_is_weak}, Propositions~\ref{prop:spat}--\ref{prop:strong}, Lemma~\ref{thm:PhiSemidiscr} leading to Theorem~\ref{thm:Phi} relies on the positivity of $L$.
The bounds~\eqref{eq:I}, \eqref{eq:J} and \eqref{eq:K} imply~\ref{eq:spat_reg}. The bounds~\eqref{eq:Holder_I}, \eqref{eq:Holder_K} imply the H\"older regularity stated in~\ref{eq:temp_reg}. 

\subsection{Proof of Theorem~\ref{thm:mild_is_weak}}
\label{sec:reg}

We write $F=R^*R+B^*LB$ and let $L$ be a mild solution, i.e., $L$ satisfies \eqref{eq:Lyapunov_mild}. Since for all $t\in\IT$, $\phi\in H$,
\begin{equation*}
  L(t)\phi
  =
  S(t)G^*GS(t)\phi
  +
  \int_0^t
    S(t-s)
    F(s)
    S(t-s)\phi
  \diff s,
\end{equation*}
we obtain for $t+h\le T$ and $\phi\in H$
\begin{align*}
  L(t+h)\phi
&=
  S(h)S(t)G^*GS(t)S(h)\phi
  +
  \int_0^{t+h}
    S(t+h-r)
    F(r)
    S(t+h-r)\phi
  \diff r\\
&=
  S(h)L(s)S(h)\phi
  +
  \int_t^{t+h}
    S(t+h-r)
    F(r)
    S(t+h-r)\phi
  \diff r.
\end{align*}
Therefore, for $\phi,\psi\in\dot H^2$
\begin{align*}
  \langle L(t+h)\phi,\psi \rangle
&=
  \langle L(t)S(h)\phi,S(h)\psi \rangle\\
&\quad
  +
  \int_t^{t+h}
    \langle
      S(t+h-r)
      F(r)
      \phi,
      S(t+h-r)
      \psi
    \rangle
  \diff r
\end{align*}
and subtracting $\langle 
    L(t)\phi,\psi 
  \rangle$ on both sides and dividing by $h>0$ gives
\begin{align*}
  \Bigg\langle 
    \frac{L(t+h)-L(t)}{h}\phi,\psi 
  \Bigg\rangle
&=
  \Bigg\langle 
    L(t)\frac{S(h)-\Id}{h}\phi,S(h)\psi 
  \Bigg\rangle
  +
  \Bigg\langle 
    L(t)\phi,\frac{S(h)-\Id}{h}\psi 
  \Bigg\rangle\\
&\quad
  +
  \frac1h
  \int_t^{t+h}
    \langle
      S(t+h-r)
      F(r)
      \phi,
      S(t+h-r)
      \psi
    \rangle
  \diff r.
\end{align*}
The semigroup $S$ is strongly differentiable, hence weakly differentiable with derivative
$
  \tfrac{\mathrm d}{\mathrm d t}
  \langle
    S(t) \phi,\psi
  \rangle
  =
  -
  \langle
    AS(t) \phi,\psi
  \rangle
$ for $\phi,\psi\in\dot H^2$.
In the limit as $h\to 0$, we obtain, by the Lebesgue differentiation theorem, the weak form~\eqref{eq:Lyapunov_var}. This completes the first direction of the proof.

Assume next that the operator-valued function $L\in\cV$ satisfies for $t\in\IT_0$, $\phi,\psi\in \dot H^2$ the variational equation \eqref{eq:Lyapunov_var}. Using~\eqref{eq:aA} we bound
\begin{align*}
  \Big|
    \frac{\dd}{\dd t}
    \langle
      L(t) \phi,\psi
    \rangle
  \Big|
&\leq
  |a(L(t)\phi,\psi)|
  +
  |a(L(t)\psi,\phi)|
  +
  |\langle
    R\phi,R\psi
  \rangle
  |
  +
  |\langle
    L(t)B\phi,B\psi
  \rangle_{\LB_2^0}
  |\\
&\leq
  \|L(t)\|_{\LB(H)}
  \big(
    \|A\phi\|\|\psi\|
    +
    \|\phi\|\|A\psi\|
  \big)\\
&\quad
  +
  \big(
    r^2
    +
    b^2
    \big\|
      A^{\frac{1-\beta}2}
      L(t)
      A^{\frac{1-\beta}2}
    \big\|_{\LB(H)}
  \big)
  \|\phi\|\|\psi\|
\end{align*}
with $t\in\IT_0$, $\phi,\psi\in \dot H^2$.
Since $\|\phi\|\lesssim\|A\phi\|=\|\phi\|_{\dot H^2}$ 
we conclude that
\begin{equation*}
  \Big|
    \frac{\dd}{\dd t}
    \langle
      L(t) \phi,\psi
    \rangle
  \Big|
  \lesssim
  (1+t^{\beta-1})
  \|\phi\|_{\dot H^2}
  \|\psi\|_{\dot H^2}
  \lesssim
  t^{\beta-1}
  \|\phi\|_{\dot H^2}
  \|\psi\|_{\dot H^2}.
\end{equation*}
By the Riesz representation theorem, there exists 
$\dot L\colon\IT_0\to \LB(\dot H^{2},\dot H^{-2})$ satisfying
\begin{equation*}
  \langle
    \dot L(t)\phi,\psi
  \rangle
  =
  \tfrac{\dd}{\dd t}
    \langle
      L(t) \phi,\psi
    \rangle
\end{equation*}
for
$
  t\in\IT_0
$
and
$
  \phi,\psi\in \dot H^2.
$
Using now the specific test functions $S(t-s)\phi$ and $S(t-s)\psi$ in the variational formulation~\eqref{eq:Lyapunov_var} yields 
\begin{equation*}
  \langle
    (
      \dot L(s) + AL(s) + L(s)A
    )S(t-s)\phi
    ,S(t-s)\psi
  \rangle
  =
  \langle
    F(L(s))S(t-s)\phi
    ,S(t-s)\psi
  \rangle.
\end{equation*}
By the product rule,
\begin{align*}
  \frac{\mathrm d}{\mathrm d s}
  \langle
    S(t-s)L(s)S(t-s)\phi,\psi
  \rangle
&=
  \langle
    (
      \dot L(s) + AL(s) + L(s)A
    )S(t-s)\phi
    ,S(t-s)\psi
  \rangle\\
&=
  \langle
    F(L(s))S(t-s)\phi
    ,S(t-s)\psi
  \rangle
\end{align*}
and therefore, integration from $0$ to $t$ yields
\begin{align*}
  \langle
    L(t)\phi,\psi
  \rangle
&=
  \langle
    S(t)L(0)S(t)\phi,\psi
  \rangle
  +
  \int_0^t
    \langle
      S(t-s)F(L(s))S(t-s)\phi,\psi
    \rangle
  \diff s\\
&=
  \Big\langle
    S(t)L(0)S(t)\phi
    +
    \int_0^t
      S(t-s)F(L(s))S(t-s)\phi
    \diff s
    ,
    \psi
  \Big\rangle.
\end{align*}
Since this identity holds for all $\psi\in \dot H^2$, the mild form of the Lyapunov equation~\eqref{eq:Lyapunov_mild} is satisfied for all $\phi\in\dot H^2$. Due to the density of $\dot H^2 \subset H$, we can approximate any $\phi \in H$ by a sequence in~$\dot H^2$ and obtain convergence of the above identity, i.e., it can be extended to elements in~$H$.

It remains to prove that \eqref{eq:Lyapunov_var} is valid for $\phi,\psi\in\dot H^\varepsilon$, $\varepsilon>0$. For this we rely on the spatial regularity Theorem~\ref{thm:Lyapunov}\ref{eq:spat_reg}. Let $\varepsilon>0$ and $\phi,\psi\in \dot H^2$. Since $L$ satisfies \eqref{eq:Lyapunov_var}, 
\begin{align*}
  \Big|
    \frac{\dd}{\dd t}
    \langle
      L(t) \phi,\psi
    \rangle
  \Big|
&\leq
  |a(L(t)\phi,\psi)|
  +
  |a(L(t)\psi,\phi)|
  +
  |\langle
    R\phi,R\psi
  \rangle
  |
  +
  |\langle
    L(t)B\phi,B\psi
  \rangle_{\LB_2^0}
  |\\
&\leq
  \|A^{1-\frac\varepsilon2}L(t)\|_{\LB(H)}
  \big(
    \|A^{\frac\varepsilon2}\phi\|\|\psi\|
    +
    \|\phi\|\|A^\frac\varepsilon2\psi\|
  \big)\\
&\quad
  +
  \big(
    r^2
    +
    b^2
    \big\|
      A^{\frac{1-\beta}2}
      L(t)
      A^{\frac{1-\beta}2}
    \big\|_{\LB(H)}
  \big)
  \|\phi\|\|\psi\|.
\end{align*}
The regularity estimate Theorem~\ref{thm:Lyapunov}\ref{eq:spat_reg} allows us to bound further
\begin{equation*}
  \Big|
    \frac{\dd}{\dd t}
    \langle
      L(t) \phi,\psi
    \rangle
  \Big|
\lesssim
  t^{\frac\varepsilon2-1}
  \|\phi\|_{\dot H^\varepsilon}
  \|\psi\|_{\dot H^\varepsilon}
\end{equation*}
for all $\phi,\psi\in \dot H^\varepsilon$.
Therefore \eqref{eq:Lyapunov_var} can be extended to all $\phi,\psi\in\dot H^{\varepsilon}$, which concludes the proof.

\section{Rational approximation of semigroups} \label{sec:appendix}

For completeness we include here a result on the error resulting from a rational approximation of a semigroup. This is~\eqref{eq:E_h_bound} in the main part of the paper. To the best of our knowledge, it is not available in the literature, but the proof is similar to those of \cite[Theorems~7.1-7.2]{thomee2006}.  

\begin{prop}\label{prop:AE}
	Let $A: \cD(A) \subset H \to H$ be a densely defined, linear, self-adjoint and positive definite operator with a compact inverse on a separable Hilbert space $H$, let $S = (S(t), t \ge 0)$ be the analytic semigroup generated by $-A$, let the function $R: \R^+ \to \R$ be given by $R(\lambda) = (1 + \lambda)^{-1}$ and let $\tau > 0$. With $t_n = \tau n$ and $S_\tau = R(\tau A)$, there exists for all $r \in[0,1]$ a constant $C_r > 0$, not depending on $A$, such that for all $\rho\in[0,2]$
	\begin{equation*}
	\|A^{\frac{r}2} (S(t_n) - S_\tau^n) \|_{\LB(H)}
	\leq
	C_r t_n^{-\frac{\rho+r}2}\tau^{\rho/2}.
	\end{equation*}
\end{prop}

\begin{proof}
	First we introduce the notation $F_n(\lambda) = R(\lambda)^n - e^{-n \lambda}$ 
	so that $S(t_n) - S_\tau^n = F_n(\tau A)$. 
	The bound of the theorem can then be written as
	\begin{equation*}
	\|A^{\frac{r}2}F_n(\tau A)\|_{\LB(H)}
	\leq
	C_r t_n^{-\frac{\rho+r}2}\tau^{\rho/2} = C_r n^{-\frac{\rho + r}{2}} \tau^{-\frac{r}{2}}
	\end{equation*}
	and since $(\tau A)^{r/2}F_n(\tau A)$ diagonalizes with respect to the eigenbasis of $A$, the claim of the theorem is equivalent to the existence of a constant $C_r > 0$ such that 
	\begin{equation*}
	|\lambda^{\frac{r}{2}} F_n(\lambda)| \le C_r n^{-\frac{\rho + r}{2}}
	\end{equation*}
	for all $\lambda \in \sigma(\tau A) \subseteq [0,\infty) $, where $\sigma(\tau A)$ denotes the spectrum of $\tau A$. 
	
	We next consider the case $\lambda \in [0,1]$. Due to the definition of $R$, there exist constants $c_1 > 0$ and $0 < c_2 < 1$ such that, for $\lambda \in [0,1]$, $|R(\lambda)-e^{-\lambda}| \le c_1 \lambda^2$ and $|R(\lambda)| \le e^{-c_2 \lambda}$. Using these bounds we get for $\lambda \le 1$ 
	\begin{align*}
	|\lambda^{\frac{r}{2}} F_n(\lambda)| &= |\lambda^{\frac{r}{2}}\left( R(\lambda) - e^{- \lambda}\right) \sum^{n-1}_{j=0} R(\lambda)^{n-1-j} e^{-j\lambda}| \\
	&\le c_1 \lambda^{\frac{r}{2}+2} \sum^{n-1}_{j=0} e^{-c_2\lambda(n-1) + (c_2-1) j \lambda} \le c_1 e^{c_2} \lambda^{\frac{r}{2}+2} n e^{-c_2 \lambda n} \\
	&= c_1 e^{c_2} n^{-1 - \frac{r}{2}} \left( \lambda n \right)^{\frac{r}{2} + 2} e^{-c_2 \lambda n} \le C_r n^{-1 - \frac{r}{2}}.
	\end{align*}
	In the last step, we have used the fact that the mapping $x \mapsto x^{r/2 + 2} e^{-c_2 x}$ is bounded on $[0,\infty)$.
	
	Next, we assume that $\lambda \in (1,\infty)$ and note that then
	\begin{equation*}
	|\lambda^{\frac{r}{2}} F_n(\lambda)| \le \left(\lambda^{\frac{r}{2n}} R(\lambda) \right)^n + \left(\lambda^{\frac{r}{2n}}e^{-\lambda}\right)^n \le \left(\lambda^{\frac{r}{2}}R(\lambda) \right)^n + \left(\lambda^{\frac{r}{2}}e^{-\lambda} \right)^n.
	\end{equation*}
	By inspecting the derivatives of the mappings $\lambda \mapsto  \lambda^{r/2} R(\lambda)$ and $\lambda \mapsto  \lambda^{r/2}e^{-\lambda}$ we see that as long as $r \in [0,1]$, they map into $(0,1/2)$ and $(0,1/e)$ respectively on $(1,\infty)$. Pick $c_3 > 0$ such that $1/2 < e^{-c_3}$. Then
	\begin{equation*}
	\left(\lambda^{\frac{r}{2}}R(\lambda) \right)^n + \left(\lambda^{\frac{r}{2}}e^{-\lambda} \right)^n < 2 e^{-c_3 n} \le C_r n^{-2}
	\end{equation*}
	for $\lambda \in (1,\infty)$ and the proof is finished. 
\end{proof}

\end{document}